\documentclass[amsppt,reqno]{amsart}

\usepackage{amsmath,amssymb} 
\usepackage{subfigure}
\usepackage{epsfig} 
\newcommand{\fracd}[2]{\displaystyle {\frac{{\displaystyle{#1}}}{{\displaystyle{#2}}}}}

\renewcommand{\descriptionlabel}[1]%
    {\hspace{\labelsep}\textit{#1}}
{%

\begin{enumerate}}%
{\end{enumerate}
}
{%

\begin{enumerate}}%
{\end{enumerate}
}

\newtheorem{definition}{Definition}[section]
\newtheorem{theorem}{Theorem}[section]

\newtheorem{lemma}{Lemma}[section]
\newtheorem{corollary}{Corollary}[section]
\newtheorem{remark}{Remark}[section]

\numberwithin{equation}{section}

\newcommand{\TV}{\operatorname*{TV}}
\newcommand{\sgn}{\operatorname*{sgn}}

\newcommand{\eps}{\varepsilon}

\def\norm#1{\|#1\|}

\def \f12{{1 \over 2}}

\title[Strongly degenerate parabolic aggregation equation]{A strongly degenerate parabolic \\ aggregation equation} 

\author[Betancourt]{F.\ Betancourt$^{\mathrm{A}}$} 
\thanks{$^{\mathrm{A}}$Departamento de Ingenier\'{\i}a Matem\'{a}tica, 
Facultad de Ciencias F\'{\i}sicas y Matem\'{a}ticas, 
Universidad de Concepci\'{o}n, 
Casilla 160-C, Concepci\'{o}n, Chile,
E-mail: {\tt  fbetanco@ing-mat.udec.cl}}
\author[B\"{u}rger]{R.\ B\"{u}rger$^{\mathrm{B}}$} 
\thanks{$^{\mathrm{B}}$CI$^{\mathrm{2}}$MA and 
Departamento de Ingenier\'{\i}a Matem\'{a}tica, 
Facultad de Ciencias F\'{\i}sicas y Matem\'{a}ticas, 
Universidad de Concepci\'{o}n, 
Casilla 160-C, Concepci\'{o}n, Chile,
E-mail: {\tt  rburger@ing-mat.udec.cl}}
\author[Karlsen]{K.\ H.\ Karlsen$^{\mathrm{C}}$}
\thanks{$^{\mathrm{C}}$Centre of Mathematics for Applications (CMA), 
  University 
 of Oslo, P.O. Box 1053, Blindern, N-0316 
 Oslo, Norway.
  E-Mail: {\tt kennethk@math.uio.no}}

\date{\today}

\begin{document}

\begin{abstract} This paper is concerned with 
a strongly degenerate convection-diffusion equation 
in one space dimension whose  convective flux 
involves a non-linear function of the total mass 
to one side of the given position. 
This equation  can be understood as a  
model of aggregation of the individuals of a population with 
the solution representing their local density. 
The  aggregation mechanism  is balanced  by a  
degenerate diffusion term accounting  for dispersal.  
In the strongly degenerate case,  solutions 
of the non-local problem  are usually 
discontinuous and need to be defined as 
weak solutions satisfying an entropy condition. 
A finite difference  scheme for 
the non-local problem is formulated and 
its convergence to the unique entropy solution is proved. 
The scheme emerges from taking divided differences 
of a monotone scheme for the 
local PDE for the primitive. 
Numerical examples illustrate the behaviour of 
entropy solutions  of the non-local problem,  
in particular the aggregation phenomenon.  
 \end{abstract}

\maketitle

\section{Introduction}
\label{sec:Introduction}

\subsection{Scope} 
This paper is related to the initial value problem 
 for a  strongly degenerate convection-diffusion equation of the form 
\begin{gather} 
\label{goveq} 
  u_t +  \left(   \Phi' \biggl( \int_{-\infty}^x  
  u(y,t) \, \mathrm{d}y  \biggr) u(x,t)
\right)_x = A(u)_{xx}, \quad x \in \mathbb{R}, 
\quad 0 < t \leq T,  \\ \label{initcond} 
 u(x,0) = u_0(x)\geq 0, \quad x \in \mathbb{R}, \quad u_0 \in (L^1\cap L^\infty)( \mathbb{R}) 
\end{gather} 
for the density $u=u(x,t)\geq 0 $, where $A(u)$ is a 
diffusion function  given by 
\begin{align} \label{adef} 
  A(u) := \int_0^u a(s) \, \mathrm{d}s, \quad \text{where $a(u) \geq 0$ for 
$u \in \mathbb{R}$}.  
\end{align} 
The model \eqref{goveq}, \eqref{initcond} was studied as a model of 
aggregation by a series of authors including   
Alt \cite{alt85}, Diaz, Nagai, and Shmarev \cite{diaz},  
 Nagai \cite{nagai1} and Nagai and Mimura \cite{nagai2,nagai3,nagai4}, all of which 
assumed that $a(u)=0$ at most at isolated values of~$u$. 
It is the purpose of this paper to study 
  \eqref{goveq}, \eqref{initcond} under the more general assumption 
  that  $a(u) =0$  on  bounded $u$-intervals 
 on which \eqref{goveq}  reduces to a first-order 
 conservation law with non-local flux.  
We assume that 
\begin{align}  \label{aincr}
 A(u) \to \infty \quad \text{as $u \to \infty$.} 
\end{align} 
This implies 
 that $a( \cdot)$ may vanish only 
on bounded subintervals  of $\mathbb{R}$.  

The key observation made in previous work 
 \cite{alt85,nagai1,nagai2,nagai3,nagai4} 
is that 
if  all coefficient functions are sufficiently smooth, and
 $u(x,t)$ is an $L^1$   solution of the 
problem \eqref{goveq}, \eqref{initcond}, then the 
 primitive (precisely, of $u(\cdot, t)$)  defined by   
\begin{align} \label{vdef} 
     v(x,t) := \int_{-\infty}^x  u(\xi, t) \, \mathrm{d}\xi, \quad t
      \in (0, T], 
\end{align} 
is a  solution  of the local initial value problem 
\begin{gather} \label{vgoveq} 
   v_t + \Phi (v)_x  =  A ( v_x )_x, \quad  x \in \mathbb{R}, \quad t
   \in (0, T], \\ 
  \label{vinit}
 v(x,0) = v_0 (x), \quad x \in \mathbb{R}, \quad 
  v_0(x) :=  \int_{-\infty}^x u_0(\xi) \, \mathrm{d} \xi .  
\end{gather}
As a non-linear but local PDE, \eqref{vgoveq} is more amenable to 
 well-posedness and numerical analysis.  In this work we use that 
 the transformation to the local equation \eqref{vgoveq} is also 
 possible in the strongly degenerate case, in which solutions 
 of \eqref{goveq} are usually discontinuous and need to be defined as 
  weak solutions. To achieve uniqueness, an additional selection
criterion is imposed, and the type of solutions sought are 
(Kru\v{z}kov-type) entropy solutions. 
 The core, and essential novelty, 
  of the paper is  the formulation and 
convergence proof of a finite difference scheme for 
 \eqref{goveq}, \eqref{initcond} (in short, ``$u$-scheme''). 
The scheme is based on a monotone 
difference scheme for the initial value problem 
 \eqref{vgoveq}, \eqref{vinit} 
 (in short, ``$v$-scheme'')  in the strongly 
degenerate case, which in turn is a special case 
 of the schemes formulated and analyzed by Evje and Karlsen 
\cite{ekdd} for the more general doubly degenerate 
equation $v_t + \Phi(v)_{x} = B(A(v_x))_x$.   The $u$-scheme 
is obtained by taking 
 finite differences of the numerical solution 
 values generated by the $v$-scheme.  

The $v$-scheme 
is, in particular, monotonicity preserving, 
 so the discrete approximations for $v$ are always monotonically 
 increasing when the  initial datum $v_0$ is, and therefore  the 
 $u$-scheme produces non-negative solutions. Moreover,  
 by modifications of standard compactness 
 and Lax-Wendroff-type arguments  it is proved 
 that the numerical  approximations 
 generated by the $u$-scheme  
 converge to the unique entropy solution of   \eqref{goveq},
 \eqref{initcond}.  An appealing feature is that 
  the  primitive  \eqref{vdef}  never needs to be 
calculated explicitly (except for the computation of 
 $v_0$). Numerical examples illustrate the behaviour of 
 entropy solutions  of 
  \eqref{goveq}, \eqref{initcond}, and  
    recorded error histories demonstrate the 
convergence of the $v$- and $u$-schemes.

\subsection{Assumptions} 
We assume that $u_0$ has compact support, 
 and that there exists a constant $\mathcal{M}$ such that 
 \begin{align} \label{initial_data_stability}  
  \TV (u_0) < \mathcal{M}. 
   \end{align} 
 We also need that $\Phi \in C^2(\mathbb{R})$, and that 
 $\Phi$ has exactly one maximum: 
\begin{align} \label{phiass} 
 \exists v^* \in \mathbb{R}: \quad \Phi ' (v^*) = 0,  \quad  
 \Phi ' (v) > 0  \; 
\text{for $v < v^*$}, \quad  \Phi ' (v)  < 0  \;
 \text{for $v > v^*$}. 
\end{align} 
This assumption is introduced 
 to facilitate some of the steps of our analysis; it is, 
 however, not essential. In fact, in our convergence analysis 
  of Section~\ref{sec:conv} we need to discuss the 
   local behaviour of the  numerical solution for~$v$ 
   close to where it includes the value $v^*$ since  that value 
    is critical in the    definition of the numerical flux. 
     If we employ a function $\Phi$ that has several separate extrema, 
      then the locations of solution values including extrema 
    are spatially well separated since the discrete 
       analogue of $v_x$ is bounded, and the techniques of 
            Section~\ref{sec:conv} can be extended to that case in a 
            straightforward manner.

\subsection{Motivation and related work} 
Equation \eqref{goveq}, or some specific cases of it,  were studied in a series of papers 
  \cite{alt85,diaz,nagai1,nagai2,nagai3,nagai4},  in all of which it is assumed that 
  $a(\cdot )$  vanishes at most at isolated 
   values of its argument, so that it is always ensured that $A' (u) >0$ for $u \geq 0$. 
  The interpretation of \eqref{goveq} as a model of the  aggregation  of  
   populations (e.g., of animals) advanced in those papers is also valid here and can 
   be illustrated as follows. Assume that $u(x,t)$ is the density of 
   the population under study, and consider the equation 
   \begin{align} \label{nmeq} 
    u_t  + \left(- k \left[ \int_{-\infty}^x u(y,t) 
   \, \mathrm{d}y -  \int_x^{\infty} u(y,t) 
   \, \mathrm{d}y\right] u \right)_x = A(u)_{xx}, \quad k> 0.  
   \end{align} 
   Here, the convective term provides a mechanism that moves 
   $u(x,t)$ to the right (respectively, to the left) if 
   \begin{align*} 
    \int_{-\infty}^x u(y,t) \, \mathrm{d}y  <  \int_x^{\infty} u(y,t) 
   \, \mathrm{d}y \quad \text{(respectively, $\ldots >  \ldots$)}.  
   \end{align*} 
In other words, an animal will move to the right (respectively, left) if the total 
 population to its right is larger (respectively, smaller) than 
   to its left. Now assume that the initial population is finite and define 
 \begin{align}  \label{Cdef} 
  C_0:= \int_{\mathbb{R}} u_0(x) \, \mathrm{d} x. 
  \end{align} 
 It is then clear that \eqref{nmeq} is an example of \eqref{goveq} if 
 $\Phi' (v) = - k(2v-C_0)$, i.e., 
 \begin{align} \label{phiez} 
 \Phi(v) = - k v(v -C_0) + \mathrm{const.}
 \end{align}  
 
 The aggregation mechanism is balanced by  
  nonlinear diffusion  described by the term $A(u)_{xx}$, 
 termed density-dependent dispersal in mathematical ecology. 
 A typical novel feature  addressed  by the present analysis 
 is  a ``threshold effect'', i.e. dispersal only 
  sets on when the density~$u$ exceeds a critical value~$u_{\mathrm{c}}>0$,   i.e. 
   \begin{align*} 
    a(u) \begin{cases} 
     =0 & \text{if $u \leq u_{\mathrm{c}}$, $u_{\mathrm{c}}>0$,} \\
      >0 & \text{if $u > u_{\mathrm{c}}$}. 
       \end{cases} 
        \end{align*}


More recently, spatially multi-dimensional aggregation equations 
 of the form 
 \begin{align} \label{mdpde} 
  u_t + \nabla \cdot (u \nabla K \ast u)  = \Delta A(u) 
  \end{align} 
  have seen an enormous amount of interest, where  the typical case treated in literature 
  is $A \equiv 0$. Here, $K$ denotes an 
  interaction potential, and $ K \ast u$ denotes spatial  convolution. For overviews we refer to  \cite{bertozzi07,lizhang10,topaz06}.  The 
    non-local and diffusive term account for long-range and 
    short-range interactions, respectively, as is 
    emphasized  in \cite{burger}. 
    The derivation of 
     \eqref{mdpde} from microscopic interacting 
     particle systems and related models, and for particular choices of~$K$ and $A$, 
     is presented in  \cite{bertozzi09,bodnar06,burger,mogilner03,morale05}.  
  Related models also include equations with fractional dissipation 
   that cannot be cast in the form \eqref{mdpde}, see  e.g.~\cite{lirodrigo09,lirodrigo09b}.

The essential research problem associated with \eqref{mdpde} (or variants of this 
equation) is the  well-posedness  of this equation together with 
bounded initial data $u(x,0) = u_0(x)$ for 
 $x \in \mathbb{R}^d$, where $d$ denotes the number of space dimensions. 
  While the short-time existence of a unique smooth solution for smooth 
   intial data is known in most situations, one wishes to 
    determine criteria in terms of the functions $K$ and $A$ (or related 
     diffusion terms), and possibly of~$u_0$, that  either ensure that  
     smooth  solutions exist globally in time, or that compel 
      that   
  solutions of  \eqref{mdpde} will blow  up in finite time.
   This problem is analyzed in \cite{bertozzi09,bertozzi07,bodnar06,burger,carrillo10,lirodrigo09,lirodrigo09b,lizhang10} (this list is far from being complete).

   The occurrence of blow-up was analyzed in terms of the properties of $K$  
    for $A \equiv 0$  in  \cite{bertozzi09,bertozzi07}; if $K$ is a radial 
    function, i.e., $K=K(|x|)$, then blow-up occurs if the Osgood condition for the 
    characteristic ODEs is violated, as occurs e.g. for $K(x) = \exp (-|x|)$, 
     while for a $C^2$ kernel this does not occur \cite{bertozzi09}. 
      Li and Rodrigo   \cite{lirodrigo09,lirodrigo09b} consider this 
       particular kernel and describe the circumstances under which 
        blow-up occurs if the aggregation equation is equipped with fractional 
        diffusion.

 The present equation \eqref{goveq} can be written as a one-dimensional version 
  of \eqref{mdpde} only in special cases. However, and as was already pointed out in  
   \cite{nagai2},   
    \eqref{nmeq} can be written as 
    \begin{align} \label{nagaieq}  
    u_t + (u \tilde{K} \ast u)_x = A(u)_{xx}
    \end{align}  
  with  the odd kernel 
    $\smash{\tilde{K} (x) = -k \sgn (x)}$.  
 Equation \eqref{nagaieq} becomes  a one-dimensional example of \eqref{mdpde} 
  if we observe that $\smash{\tilde{K} \ast u  = K' \ast u}$, where $K'$ denotes the 
  derivative of~$K$, if we choose the even kernel 
        \begin{align} \label{ourkernel}
     K(x)= -k |x| + C,   
     \end{align} 
    where  $C$ is a constant. We can write this as $K(x)=- \kappa (|x|)$ for $\kappa(r) = r-C$. 
     Suppose that one uses this kernel in the multi-dimensional equation 
      \eqref{mdpde}. It is then straightforward to verify  
       that 
      in absence of dispersal ($A \equiv 0$), the kernel \eqref{ourkernel} 
    satisfies the integral condition for blow-up  in finite time, see \cite{bertozzi09}. 
    One result of our analysis is then that a strongly degenerate diffusion term $A(u)_{xx}$, accounting 
     for dispersal, is sufficient  to prevent blow-up of solutions of \eqref{mdpde} 
      provided that the condition \eqref{aincr} is satisfied.  
       In fact, in  the context of aggregation models that are based either on 
       \eqref{goveq} or on the more recently studied equation \eqref{mdpde}, 
        the present work is the first that incorporates  a strongly degenerate 
         diffusion term, i.e. involves a function $A(u)$ that is flat on a 
         $u$-interval of positive length.  
         So far,  diffusion terms that have been 
           considered in \eqref{goveq} degenerate at most at isolated $u$-values. 
        Nagai and Mimura \cite{nagai2} studied the Cauchy problem for
 equation \eqref{goveq} 
 under the assumptions $A(0) =0$, $A'(u) >0$ being an 
odd function. The initial function for the Cauchy problem
in \cite{nagai2} is assumed to be bounded, non-negative and 
 integrable.  They prove existence and uniqueness of a bounded and 
continuous  solution to the initial-value problem.  In \cite{nagai3} the asymptotic behaviour 
 of solutions to the same problem was studied for  the specific choice 
  \begin{align} \label{aum} 
   A(u) = u^m, \quad m>1. 
   \end{align} 
 It seems that the analysis of \eqref{mdpde} with  degenerate diffusion has just started. 
  Li and Zhang   \cite{lizhang10} study this equation in one
  space dimension for the  diffusion function $A(u) = u^3/3$, 
   which degenerates at $u=0$ only.  On the other hand, 
      the numerical simulations presented herein show that 
       under strongly degenerate diffusion,  typical features of the 
        aggregation phenomenon such as ``clumped'' solutions 
         with very sharp edges \cite{topaz06} appear.

 Finally, we comment that 
%
%
%
 there  is also  considerable interest in 
       the well-posedness and other properties of 
  the local PDE  \eqref{vgoveq} (or variants of this equation)  
    under the assumption  that $A(\cdot)$ is an increasing but 
     bounded function, an effect usually denoted by {\em saturating diffusion},  
     cf.~\cite{chertock05} and references cited in that paper. 
     In this case, which is explicitly excluded by our assumption 
        \eqref{aincr}, $v_x$ (though not necessarily $v$ itself) becomes in general 
        unbounded.  It is at present unclear whether well-posedness of 
         our nonlocal problem \eqref{goveq}, \eqref{initcond} can also be achieved 
          under saturating diffusion.

\subsection{Outline of the paper}  
The remainder of this paper
is organized as follows. In Section~\ref{sec:def} we state 
the definition of an entropy solution of \eqref{goveq}, \eqref{initcond}, 
and point out that an entropy solution is also a weak solution.  
In Section~\ref{subsec:rh} we state jump conditions that can 
be derived from the definition of an entropy solution, 
and in Section~\ref{subsec:uniqueness} we prove 
the uniqueness of an entropy solution. 
  
   Section~\ref{sec:conv}  
presents a convergence analysis for the $u$-scheme, which in 
  part relies on standard compactness properties for the $v$-scheme. 
   In Section~\ref{subsec:prel}, the schemes are described. 
Section~\ref{subsec:v_estimates}
     contains a series of lemmas  stating uniform estimates 
      on the numerical approximations generated by the $v$- and the $u$-schemes, 
       which allow to employ standard compactness arguments to  
         deduce that both schemes converge. 
The final convergence result (Theorem~\ref{th4.1}) 
                and its proof are presented in Section~\ref{subsec:convent}. 
                  This proof involves 
                  a discrete cell entropy 
                    satisfied by the $u$-scheme, which 
                     eventually permits to conclude that  $\smash{\{u_{j-1/2}^n\}}$  
                      converges to an entropy solution as the discretization parameters tend to zero. 
                      This means, in particular, that an entropy solution exists. 
                The mathematical model and the $v$- and $u$-scheme are illustrated by  numerical 
                examples presented in Section~\ref{sec:num}. 

\section{Definition of an entropy solution}  \label{sec:def} 
\begin{definition} \label{def:ent}  
A measurable, non-negative  function $u$ is an {\em entropy  solution} of the 
 initial value problem   \eqref{goveq}, \eqref{initcond} if it satisfies the 
  following conditions:

  \begin{enumerate} 
  \item We have 
   $u \in L^{\infty} ( \Pi_T) \cap L^1 (\Pi_T) \cap L^{\infty} (0,T; BV(\mathbb{R})) \cap C(0, T; L^1 (\mathbb{R}))$, and $A(u) \in L^2(0,T;H^1(\mathbb{R}))$. 
   \item The initial condition \eqref{initcond} is satisfied in the 
    following sense: 
    \begin{align}  \label{esslimcond} 
     \lim_{t \downarrow 0}  \int_{\mathbb{R}} 
      \bigl| u(x,t) - u_0(x) \bigr| \, \mathrm{d}x =0.  
     \end{align} 
  \item  For all non-negative test 
   functions $\varphi \in C_0^{\infty} ( \Pi_T)$, the entropy inequality
   \begin{align}  \label{entrineq} 
   \begin{split} 
    \forall k \in \mathbb{R}: 
   \iint_{\Pi_T} & \Bigl\{ 
    |u -k| \bigl( \varphi_t + \Phi' (v) \varphi_x \bigr) - \sgn (u-k)  uk \Phi'' (v) \varphi
    \\     & 
     + \bigl| A(u) - A(k) \bigr|  \varphi_{xx}  \Bigr\} \, \mathrm{d}x \, \mathrm{d}t \geq 0 
     \end{split}  
  \end{align} 
   is satisfied, 
   where $v(x,t)$ is defined by \eqref{vdef} and $\Pi_T:=\mathbb{R}\times (0,T)$. 
      \end{enumerate} 
\end{definition}

 \begin{definition} \label{def:weak}  A measurable function $u$ is said to be a {\em weak solution} of the 
 initial value problem   \eqref{goveq}, \eqref{initcond} if it 
  satisfies items (1) and (2) of Definition~\ref{def:ent} and if  
  the following equality is satisfied for all  test 
   functions $\phi \in C_0^{\infty} ( \Pi_T)$:
   \begin{align}  \label{weakeq} 
   \begin{split} 
   \iint_{\Pi_T} & \Bigl\{ 
     u  \bigl( \phi_t + \Phi' (v) \phi_x \bigr)
         + A(u) \phi_{xx} \Bigr\} \, \mathrm{d}x \, \mathrm{d}t = 0.
     \end{split}  
  \end{align} 
\end{definition}

It is straightforward to check that an entropy solution  
 of the 
 initial value problem   \eqref{goveq}, \eqref{initcond}
  is, in particular, a weak solution.

\begin{lemma} Assume that $u$ is an entropy solution 
 of the 
 initial value problem   \eqref{goveq}--\eqref{initcond} (cf.\
  Definition~\ref{def:ent}). Then~$u$ is a weak solution 
(cf.\  
  Definition~\ref{def:weak}).
\end{lemma} 

\begin{proof} Choosing $\smash{k \geq \| u\|_{L^{\infty} (\Pi_T)} }$ in 
 \eqref{entrineq} we obtain 
 \begin{align*} 
  \iint_{\Pi_T} \Bigl\{ -(u-k) 
  \bigl( \phi_t + \Phi' (v) \phi_x \bigr)
         - A(u) \phi_{xx} \Bigr\} \, \mathrm{d}x \, \mathrm{d}t 
\geq - k  \iint_{\Pi_T} u\Phi'' (v) \phi \, \mathrm{d}x \, \mathrm{d}t  
\end{align*} 
or equivalently, 
\begin{align}  \label{ineq2.11} 
\begin{split} 
& \iint_{\Pi_T} \Bigl\{ u 
  \bigl( \phi_t + \Phi' (v) \phi_x \bigr)
         +A(u) \phi_{xx} \Bigr\} \, \mathrm{d}x \, \mathrm{d}t \\ 
& \leq  k  \iint_{\Pi_T}\Bigl\{ \phi_t 
 + \bigl(  \Phi' (v) \phi\bigr)_x \Bigr\}  \, \mathrm{d}x \, \mathrm{d}t   =0. 
 \end{split} 
\end{align} 
On the other hand, since we look for non-negative solutions, it suffices to set $k=0$ 
in \eqref{entrineq} to deduce that we always have 
\begin{align} \label{ineq2.12} 
 \iint_{\Pi_T} \Bigl\{ u 
  \bigl( \phi_t + \Phi' (v) \phi_x \bigr)
         + A(u) \phi_{xx} \Bigr\} \, \mathrm{d}x \, \mathrm{d}t  \geq 0. 
         \end{align} 
Combining \eqref{ineq2.11} and \eqref{ineq2.12} we see that $u$ satisfies 
\eqref{weakeq}.          
\end{proof} 

\section{Jump conditions and uniqueness} \label{sec:jumpuniq} 

\subsection{Rankine-Hugoniot condition and entropy jump condition} \label{subsec:rh} 
 Assume that~$u$ is an entropy solution  having a discontinuity 
 at a point $(x_0, t_0) \in \Pi_T$ 
  between the approximate limits~$u^+$ and~$u^-$  of~$u$ 
  taken with respect to $x>x_0$ and $x<x_0$, respectively.  
  Standard results from the theory of 
  entropy solutions of  strongly degenerate parabolic equations 
  imply that such a discontinuity is possible only if 
  $A(u)$ is flat for 
  $u \in  \mathcal{I}(u^-, u^+) := [\min\{u^-, u^+\}, 
   \max \{ u^-, u^+\}]$. In that case, the propagation 
    velocity of the jump is given by the Rankine-Hugoniot 
    condition, which is derived by standard arguments from
     the weak formulation  \eqref{weakeq}: 
    \begin{align}  \label{rhprel} 
     s = \frac{1}{u^+-u^-} \Bigl( \Phi'  (v^+) u^+ - \Phi' (v^-) u^- 
      - \bigl( A(u)_x \bigr)^+ + \bigl( A(u)_x \bigr)^- \Bigr). 
      \end{align} 
   Here,   $\smash{( A(u)_x )^+}$ and 
    $\smash{( A(u)_x )^- }$  denote the 
      approximate limits of~$A(u)_x$ 
  taken with respect to $x>x_0$ and $x<x_0$, respectively, 
   and $v^+$ and $v^-$ denote the corresponding limits of  
  $v(x, t)$. However, since $v(x,t)$ is continuous, we actually have 
   $v^+=v^-$, and the Rankine-Hugoniot condition \eqref{rhprel} 
   reduces to 
   \begin{align} \label{sdef} 
     s = \Phi' \bigl(v(x_0, t_0 )\bigr) 
      - \frac{( A(u)_x )^+ -( A(u)_x )^- }{u^+-u_-}. 
      \end{align} 
      In  addition, a discontinuity between two solution values needs
      to satisfy the jump entropy condition
      \begin{align} \label{ejump1}  \begin{split} 
    &   \forall k \in  \bigl( \min\{u^-, u^+\}, 
   \max \{ u^-, u^+\} \bigr): \\
   &  \fracd{\Phi'(v^+) (u^+-k)- (A(u)_x)^+}{u^+-k} 
    \leq s \leq \fracd{\Phi'(v^-) (u^--k)- (A(u)_x)^-}{u^--k}.
    \end{split} 
   \end{align} 
   Taking into account $v^+=v^-$ and \eqref{sdef}, this reduces to 
   \begin{align} \label{ejump2}  \begin{split} 
     &  \forall k \in \bigl( \min\{u^-, u^+\}, 
   \max \{ u^-, u^+\} \bigr): \\&  
      \fracd{(A(u)_x)^+}{u^+-k} 
    \geq   \frac{( A(u)_x )^+ -( A(u)_x )^- }{u^+-u_-}   \geq \fracd{(A(u)_x)^-}{u^--k}.
    \end{split} 
   \end{align} 
   In particular, if $A(\cdot)$ is flat on an open interval 
   containing $ \mathcal{I}(u^-, u^+)$, then the double inequality 
    \eqref{ejump2} is trivially satisfied. That  $\Phi(v)$ is smooth, 
     greatly simplifies the jump and entropy jump conditions.

\subsection{Uniqueness of entropy solutions}  \label{subsec:uniqueness}  
The uniqueness of entropy solutions is an immediate consequence of 
a result proved in \cite{kr03} (cf.~also \cite{Chen:2005wf}) 
regarding continuous dependence of entropy 
solutions with respect to the flux function. More precisely, we have 

\begin{theorem}
There exists at most one entropy solution (according to Definition\ref{def:ent})
of the  initial value problem \eqref{goveq}, \eqref{initcond}. Moreover, 
there exists a constant $C=C\left(\max |\Phi'|\right)$ such that
\begin{align*}
\bigl\| u(\cdot,t)-\bar u(\cdot,t) \bigr\|_{L^1(\mathbb{R})}
\le C \norm{u_0-\bar u_0}_{L^1(\mathbb{R})}, \qquad 
\forall t \in (0,T],
\end{align*}
where $u$ and $\bar u$ are entropy solutions with 
initial data $u_0$ and $\bar{u}_0$, respectively. 
\end{theorem}
 
\begin{proof} Let $u$ be an entropy solution of the problem 
$$
u_t +  \bigl(V(x,t)u\bigr)_x =  A(u)_{xx}, \qquad 
V(x,t):=\Phi'\left(\int_{-\infty}^x  u(y,t) \, \mathrm{d}y\right),
$$
with initial data $u(0,x)=u_0(x)$, and let $\bar u$ be an 
entropy solution of the problem 
$$
\bar u_t +  \left(\bar V(x,t) \bar u\right)_x = A(\bar u)_{xx}, 
\qquad \bar V(x,t):=\Phi'\left(\int_{-\infty}^x  \bar u(y,t) \, \mathrm{d}y\right).
$$
with initial data $\bar u(0,x)=\bar u_0(x)$.  
According to \cite{Chen:2005wf,kr01}, 
keeping in mind that $u$ and $\bar u$ are of bounded 
variation, i.e., $u, \bar{u} \in L^\infty(0,T;BV(\mathbb{R}))$, 
there exists a constant $C$ such that 
\begin{align*}
\bigl\| u(\cdot,t)-\bar u(\cdot,t) \bigr\|_{L^1(\mathbb{R})}
& \le \norm{u_0-\bar u_0}_{L^1(\mathbb{R})}
 +\int_0^t \bigl|V_x(x,s)-\bar V_x(x,s) \bigr|\, \mathrm{d}s
\\  & \qquad\qquad\qquad 
+\int_0^t \bigl|V(x,s)-\bar V(x,s) \bigr|   \TV (u(\cdot,s)) \, \mathrm{d}s
\\ & \le \norm{u_0-\bar u_0}_{L^1(\mathbb{R})}
+ C\int_0^t \bigl|V_x(x,s)-\bar V_x(x,s) \bigr|\, \mathrm{d}s.
\end{align*}
Observe that 
$$
\int_0^t \bigl|V_x(x,s)-\bar V_x(x,s) \bigr|\, \mathrm{d}s 
\le \max |\Phi'| \int_0^t \bigl|u(x,s)-\bar u(x,s) \bigr|\, \mathrm{d}s,
$$
so that by the Gronwall inequality we arrive at
\begin{align*}
\bigl\| u(\cdot,t)-\bar u(\cdot,t) \bigr\|_{L^1(\mathbb{R})}
\le \exp{\left(\max |\Phi'|\, t \right)} 
\norm{u_0-\bar u_0}_{L^1(\mathbb{R})}.
\end{align*}
\end{proof}

\section{Convergence analysis of numerical schemes}  \label{sec:conv} 

\subsection{Preliminaries} We define  the vectors  \label{subsec:prel} 
 $\smash{U^n := \{u_{j+1/2} ^n\}_{j \in \mathbb{Z}}}$ and 
 $\smash{V^n := \{v_j^n \}_{j \in \mathbb{Z}}}$,  
and discretize $\mathbb{R}$ by $x_j := j \Delta x$, 
$j\in\mathbb{Z}$, and the time interval $[0,T]$ by $t_n = n \Delta t$, $n=0, \dots, N$, 
 $\Delta t:= T/N$, $N \in \mathbb{N}$. We denote by $\smash{u_{j+1/2}^n}$ the 
cell average over  $\smash{I_j:=[x_{j},x_{j+1}]}$ at time 
$t_n$ and $j\in\mathbb{Z}$. We also define  $\lambda := \Delta t / \Delta x$ and  
 $\smash{\mu := \Delta t/ \Delta x^2 
 = \lambda / \Delta x}$ and wherever convenient 
 use the spatial difference operators
 $\smash{\Delta_+ \phi_j := \phi_{j+1} -  \phi_j}$, 
 $\smash{\Delta_- \phi_j := \phi_j - \phi_{j-1}}$, and 
\begin{align*} 
 \Delta^{\! 2} \phi_j :=  \Delta_+ \Delta_- \phi_j
= \phi_{j+1} - 2 \phi_j + \phi_{j-1}.
\end{align*}

We assume that the 
 initial datum $u_0$ is discretized  via 
\begin{align} 
  u_{j+1/2}^0 := \frac{1}{\Delta x} \int_{I_j} u_0(\xi) \, d \xi, \quad 
 j\in\mathbb{Z}.  
\end{align} 
Moreover,  we define the 
operator $\mathcal{S}_{\Delta x}$ and its inverse $\smash{\mathcal{S}_{\Delta x}^{-1}}$ via 
\begin{align} \label{sdx} 
 \mathcal{S}_{\Delta x} (U^n; j) := \Delta x \sum_{l=-\infty}^{j-1}
 u_{l+1/2}^n, \quad 
  \mathcal{S}^{-1}_{\Delta x} (V^n; j) := \frac{v_{j+1}^n -
   v_j^n}{\Delta x}. 
\end{align} 
Clearly,   $\smash{ \mathcal{S}_{\Delta x}}$ and 
 $\smash{ \mathcal{S}_{\Delta x}^{-1}}$
   are the discrete analogues of the 
 integral and differential operators that convert~$u(\cdot, t_n)$ into~$v(\cdot, t_n)$ and vice versa, respectively.  Since we assume that~$u_0$ is compactly supported, the sum  in \eqref{sdx} is actually finite. 

The numerical scheme for 
 the initial value problem  \eqref{goveq}, \eqref{initcond} can be compactly written as 
follows: 
\begin{align} \label{comb_scheme}
 U^{n+1} = \bigl[ \mathcal{S}_{\Delta x}^{-1} \circ 
 \mathcal{H} \circ   \mathcal{S}_{\Delta x} \bigr] U^n, 
 \quad n=0, \dots, N-1,  
\end{align} 
where the basic idea is to utilize a standard scheme of the form 
\begin{align} \label{vscheme1} 
 V^{n+1} = \mathcal{H} (V^n) , \quad n=0,  \dots, N-1
\end{align} 
for approximate solutions of the local PDE  
\eqref{vgoveq}, starting from the initial data 
\begin{align*} 
 v_j^0 := \Delta x \sum_{l=- \infty}^{j-1} u_{l+1/2}^0 = \int_{-\infty}^{x_j} u_0(\xi) 
  \, \mathrm{d} \xi, \quad j \in \mathbb{Z}. 
 \end{align*} 
 Clearly, if  $C_0$ is the total mass defined in \eqref{Cdef}, then we have that 
 \begin{align} 
0 \leq v_j^0 \leq C_0, \quad v_j^0 \leq v_{j+1}^0 \quad \text{for all $j \in \mathbb{Z}$.} 
\label{monotinit} 
\end{align} 
Let us emphasize here that \eqref{comb_scheme} implies that
\begin{align*}   
 U^{n} = \bigl[ \mathcal{S}_{\Delta x}^{-1} \circ 
 \mathcal{H} \circ   \mathcal{S}_{\Delta x} \bigr]^n  U^0 
 = \bigl[ \mathcal{S}_{\Delta x}^{-1} \circ 
 \mathcal{H}^n \circ   \mathcal{S}_{\Delta x} \bigr]  U^0.
 \end{align*} 
This means that for the actual computation of $U^n$ from $U^0$, the  operators 
 $    \mathcal{S}_{\Delta x}$ and $\smash{ \mathcal{S}_{\Delta x}^{-1}}$
 need to be applied only once, and not for every time step.

To derive properties  of the  scheme \eqref{comb_scheme}, 
 we first analyze the scheme \eqref{vscheme1}, which is here given by 
 the marching formula 
\begin{align} \label{marching} 
 v_j^{n+1} = v_j^n - \lambda \Delta_+  \bigl[ 
 h \bigl( v_{j-1}^n ,  v_j^n \bigr) - A \bigl(\Delta_-
   v_j^n / \Delta x  \bigr) \bigr], 
\quad j\in\mathbb{Z}, \quad n=0,1, 2, \dots, 
\end{align}  
where $\lambda$ is  subject to
the CFL condition stated below, and  
\begin{align}  \label{eoflux} 
 h(w,z) := \Phi(0) + \Phi_+ (w) 
 + \Phi_- (z)
\end{align}  
is 
the Engquist-Osher flux \cite{eopaper},  
where  we define the functions
 \begin{align} \label{phipmdef}  
  \Phi_+ (v) := \int_0^v \max \bigl\{ 0, \Phi' (s) \bigr\} \, \mathrm{d}s, 
   \quad 
\Phi_- (v) := \int_0^v \min \bigl\{ 0, \Phi' (s) \bigr\} \, \mathrm{d}s.
\end{align}  
 We assume that $\Delta t$ and $\Delta x$ satisfy the 
 CFL stability condition 
\begin{align} \label{cflcond} 
2\lambda  \max_{u \in \mathbb{R}}       \bigl| \Phi'(u) \bigr|  
 +  2\mu \max_{u \in \mathbb{R}}      \bigl| a (u) \bigr|   \leq 1. 
\end{align} 
Note that the scheme for~$u$ can be written as 
\begin{align} \label{scheme100} 
 u_{j+1/2}^{n+1} = u_{j+1/2}^n - \lambda \Delta_+ G_j^n + \mu \Delta^2  
  A \bigl( u_{j+1/2}^n \bigr), \quad j \in \mathbb{Z}, 
\quad n=0,1,2 , \dots, 
\end{align} 
where we define 
\begin{align} 
 G_j^n := \fracd{1}{\Delta x } \Delta_+ h \bigl( v_{j-1}^n, v_j^n \bigr)
= 
\label{scheme101} 
    \fracd{1}{\Delta x} 
 \left( \int_{v_{j-1}^n}^{v_j^n} 
  \Phi_+' (s)  \, \mathrm{d}s + 
\int_{v_{j}^n}^{v_{j+1}^n} 
  \Phi_-' (s)  \, \mathrm{d}s \right).  
  \end{align} 

For the ease of reference, we will refer to \eqref{marching}--\eqref{phipmdef} 
 and \eqref{scheme100}, \eqref{scheme101} as ``$v$-scheme'' and ``$u$-scheme'', respectively. 

\subsection{Uniform estimates on $\{v_j^n\}$ and $\{u_j^n\}$} \label{subsec:v_estimates} 

\begin{lemma}\label{scheme_for_v_is_monotone}  Under the CFL condition \eqref{cflcond}, the 
 $v$-scheme defined by \eqref{marching}--\eqref{phipmdef}  is monotone. 
\end{lemma}

\begin{proof} We rewrite  the scheme \eqref{marching} 
as 
\begin{align*} 
 v_j^{n+1} = \mathcal{H}  \bigl( v_{j-1}^{n},  v_{j}^{n},   v_{j+1}^{n}
\bigr) =: \mathcal{H}_j^n , \quad j\in\mathbb{Z}, \quad n=0,1, \dots, N-1.  
\end{align*} 
Since $a \geq 0$, we then have  
\begin{align*} 
\frac{\partial \mathcal{H}_j^n}{\partial v_{j\pm1}^n} & = 
\mp  \lambda \min \bigl\{ 0, \Phi' \bigl( v_{j\pm1}^n \bigr) 
 \bigr\} + \mu 
  a \bigl( \Delta_{\pm} v_{j}^n / \Delta x \bigr)
  \geq 0,  
\end{align*} 
while  the CFL condition \eqref{cflcond} implies that 
\begin{align*} 
\frac{\partial \mathcal{H}_j^n}{\partial v_{j}^n}&  =  
 1 - \lambda \bigl(  \max \bigl\{ 0, \Phi' \bigl( v_{j}^n \bigr)
 \} - \min   \bigl\{ 0, \Phi' \bigl( v_{j}^n \bigr)
 \} \bigr)  - \mu \Delta_+ 
  a \bigl( \Delta_-  v_{j}^n/ \Delta x \bigr)  \\
  & = 1 - \lambda \bigl| \Phi' \bigl( v_{j}^n \bigr) \bigr| - \mu \Delta_+ 
  a \bigl( \Delta_-  v_{j}^n/ \Delta x \bigr) 
\geq 0. 
\end{align*}
\end{proof}

As a monotone scheme, the scheme \eqref{marching} is total variation diminishing (TVD)  and monotonicity 
preserving. Since \eqref{marching} represents an explicit three-point scheme, 
for a fixed discretization $(\Delta x, \Delta t)$ we will always 
 have 
\begin{align} \label{wa}  
 v_j^n =0 \quad \text{for $j < -\mathcal{K}$}, \quad  
  v_j^n= C_0 \quad \text{for $j >\mathcal{K}$}
  \end{align}  
  for a sufficiently large constant $\mathcal{K}>0$. 
Thus, we can state the following corollary. 

\begin{corollary} \label{crol4.1}  If \eqref{monotinit}  and the CFL condition 
 \eqref{cflcond} hold,   then 
 the numerical solution $\{v_j^n\}$ produced by the $v$-scheme \eqref{marching}--\eqref{phipmdef}
satisfies 
\begin{align} 
0 \leq v_j^n \leq C_0, \quad v_j^n \leq v_{j+1}^n \quad \text{\em for all $j \in \mathbb{Z}$, $n=1, 
 \dots, N$.} 
\label{monotalways} 
\end{align} 
As a  direct consequence, 
 the numerical solution values $V^n = \smash{\{ v_j^n \}_{j \in \mathbb{Z}}}$ 
satisfy the (trivial)  uniform total variation bound    
\begin{align*}
\TV (V^n) = 
 \sum_{j\in\mathbb{Z}}\left|v^{n}_{j+1}-v^{n}_{j}\right|=C_0.
\end{align*}
\end{corollary}

\begin{lemma}\label{LC_respect_t}
The  numerical solution $\smash{ \{ v_j^n \}}$  
produced by the $v$-scheme \eqref{marching}--\eqref{phipmdef}
satisfies the 
 $L^1$ Lipschitz continuity in time property, i.e., 
 there exists a constant~$C_1$, which is  independent of $\Delta : = (
 \Delta x, \Delta t)$,  such that  
\begin{align} \label{ineq3.12} 
 \sum_{j\in\mathbb{Z}}\left| v^{n+1}_{j}-v^{n}_{j}\right|\leq
 C_1 \lambda.
\end{align}
\end{lemma}
\begin{proof} For $j \in \mathbb{Z}$, the 
 quantity 
$\smash{ w^{n+1/2}_{j}:=v^{n+1}_{j}-v^{n}_{j}}$
 satisfies   
\begin{align} \label{eq4.14} \begin{split} 
 w^{n+3/2}_{j}-w^{n+1/2}_{j}&= -\lambda \Delta_+ \bigl[ h
 \bigl(v^{n+1}_{j-1},v^{n+1}_{j}\bigr)-h \bigl(v^{n}_{j-1},v^{n}_{j}\bigr)\bigr] \\
& \quad + \lambda \Delta_+ \bigl[A\bigl( \Delta_- v^{n+1}_{j} / \Delta
   x \bigr) - A\bigl( \Delta_- v^{n}_{j} / \Delta
   x \bigr) \bigr].  
   \end{split} 
\end{align}
We define 
\begin{align*} 
 \theta (s) := \begin{cases} 1/s & \text{if $s \neq 0$,} \\
 0 & \text{otherwise,} 
 \end{cases} 
 \end{align*} 
and  the  quantities
\begin{align} \label{considerq} \begin{split} 
 B^{n+1/2}_j &:= 
 \bigl[ h\bigl(v^{n}_{j-1},v^{n+1}_{j}\bigr)-h\bigl(v^{n}_{j-1},v^{n}_{j}\bigr) \bigr] \theta\bigl(v^{n+1}_{j}-v^{n}_{j}
  \bigr), \\
  C^{n+1/2}_j &:= 
 \bigl[ h \bigl(v^{n+1}_{j},v^{n+1}_{j+1}\bigr)-h\bigl(v^{n}_{j},v^{n+1}_{j+1}\bigr)
  \bigr] \theta \bigl( v^{n+1}_{j}-v^{n}_{j} \bigr), \\ 
  D^{n+1/2}_j &:=
 \bigl[ A \bigl(\Delta_+ v^{n+1}_{j} / \Delta x \bigr)-
 A  \bigl(\Delta_+ v^{n}_{j} / \Delta x \bigr) \bigr]   
 \theta \bigl(  \Delta_+ v^{n+1}_{j}-  \Delta_+ v^{n}_{j} \bigr) .  
 \end{split} \end{align}
Due to the monotonicity  of $\smash{ \{v_j^n\}}$ (see \eqref{monotalways}) we have 
\begin{equation}\label{sing_of_coeff}
 C^{n+1/2}_j \geq 0, \quad  D^{n+1/2}_j \geq 0 , \quad B^{n+1/2}_j\leq 0.  
\end{equation}
After some manipulations and using \eqref{monotalways} we obtain from \eqref{eq4.14} 
\begin{align*}
w^{n+3/2}_{j}&=w^{n+1/2}_{j}\bigl[ 1-\lambda C^{n+1/2}_j+ \lambda B^{n+1/2}_j- 
\lambda \bigl( D^{n+1/2}_{j-1}+D^{n+1/2}_{j}\bigr)\bigr]  \\
& \quad + w^{n+1/2}_{j-1} \lambda \bigl(C^{n+1/2}_{j-1}+ D^{n+1/2}_{j-1}\bigr) 
 + w^{n+1/2}_{j+1} \lambda \bigl( - B^{n+1/2}_{j+1}+ D^{n+1/2}_{j}\bigr).
\end{align*}
Using the CFL condition we find 
\begin{align*} 
\bigl|w^{n+3/2}_{j}\bigr|&\leq \bigl|w^{n+1/2}_{j}\bigr|
\bigl[ 1-\lambda \bigl( C^{n+1/2}_j-   B^{n+1/2}_j + D^{n+1/2}_{j-1}
+D^{n+1/2}_{j}\bigr)\bigr]  \\
&\quad +\bigl|w^{n+1/2}_{j-1}\bigr| \lambda \bigl(  C^{n+1/2}_{j-1}+ D^{n+1/2}_{j-1}\bigr) 
+ \bigl|w^{n+1/2}_{j+1}\bigr| \lambda \bigl(- B^{n+1/2}_{j+1}+
D^{n+1/2}_{j}\bigr).
\end{align*}
Summing this over $j\in\mathbb{Z}$, using \eqref{sing_of_coeff} and \eqref{monotalways} we obtain
\begin{align*}
\sum_{j\in\mathbb{Z}}\bigl|w^{n+3/2}_j\bigr|\leq \sum_{j\in\mathbb{Z}}\bigl|w^{n+1/2}_j\bigr|,
\end{align*}
which implies that 
\begin{align*}
\sum_{j\in\mathbb{Z}}\bigl|w^{n+3/2}_j\bigr|\leq \sum_{j\in\mathbb{Z}}\bigl|w^{1/2}_j\bigr|.
\end{align*}
From \eqref{marching} with $n=0$ we get 
\begin{align*}
 \sum_{j\in\mathbb{Z}}\bigl|w^{1/2}_j\bigr| &=\sum_{j\in\mathbb{Z}}\bigl|v^{1}_j-v^{0}_j\bigr| 
 =\sum_{j\in\mathbb{Z}} \lambda 
  \bigl| \Delta_+ \bigl(  
 h \bigl( v_{j-1}^0 ,  v_j^0 \bigr) - A \bigl(\Delta_-
   v_j^0 / \Delta x  \bigr) \bigr) \bigr|.
 \end{align*} 
Using \eqref{initial_data_stability}   we
 arrive at  \eqref{ineq3.12}. 
 \end{proof}

\begin{lemma}\label{L_infinity_u} 
The  numerical solution $\smash{ \{ v_j^n \}}$  
produced by the $v$-scheme \eqref{marching}--\eqref{phipmdef}
satisfies the inequality 
    $\smash{| \Delta _+ v^{n}_{j} / \Delta x |\leq C_3}$
     with a constant~$C_3$, which is  independent of~$\Delta$.  Equivalently,  the 
   solution  $\smash{ \{ u_{j+1/2}^n \}}$  
 generated by the $u$-scheme 
  \eqref{scheme100}, \eqref{scheme101} satisfies the 
  uniform  $L^{\infty}$ bound 
    \begin{align} 
      \bigl|  u_{j+1/2}^n \bigr| \leq C_3 \quad 
       \text{\em for all $j \in \mathbb{Z}$, $n =0, \dots, N$.} 
    	\end{align} 
\end{lemma}

\begin{proof}
It is sufficient  to show that 
$\smash{A ( \Delta _+ v^{n}_{j} / \Delta x )   \leq C_2}$  
 for a   constant~$C_2$ that  is  independent of~$\Delta$.  
Taking into account   \eqref{wa} we get  
\begin{align*}
 &  \bigl|A \bigl(   \Delta _+ v^{n}_{j} / \Delta x \bigr) \bigr|-
 \bigl|h \bigl(v^{n}_{j},v^{n}_{j+1}\bigr)\bigr| \\&\leq
 \bigl|A \bigl(   \Delta _+ v^{n}_{j} / \Delta x \bigr)- h\bigl(v^{n}_{j},v^{n}_{j+1}\bigr)\bigr| \\
&= \left|\Phi (0) + \sum^{j}_{k=-\infty}	  \Delta_- \bigl( 
A \bigl(   \Delta _+ v^{n}_{k} / \Delta x \bigr)- h\bigl(v^{n}_{k},v^{n}_{k+1}\bigr)\bigr) 
\right|\\
&=\left|\sum^{j}_{k=-\infty}\frac{v^{n+1}_{k}-v^{n}_{k}}{\lambda}+\Phi(0)\right|
 \leq \fracd{1}{\lambda} \sum_{k \in \mathbb{Z}} \bigl| v^{n+1}_{k}-v^{n}_{k} \bigr| + 
  \bigl|\Phi(0) \bigr|.
\end{align*}
Due to   Lemma~\ref{LC_respect_t}, we see that  
$\smash{ |A( \Delta_+ v^{n}_{j} / \Delta x ) | \leq C_2}$
 if we choose $C_2=C_1+ |\Phi(0) |$.
 Taking into account \eqref{aincr}  concludes the proof. 
\end{proof}

\begin{remark} 
 Under all the previous assumptions we see that 
 the $v$-scheme \eqref{marching}--\eqref{phipmdef} satisfies  
 the following discrete cell entropy inequality
\begin{align*}
\left|v^{n+1}_{j}-k\right|-\left|v^{n}_{j}-k\right|+\lambda \Delta_+ H^{n}_{j-1/2}
 -\lambda  \sgn \bigl(v^{n+1}_{j}-k\bigr)
  \Delta_+ A\bigl(\Delta_-v^{n}_{j} / \Delta x\bigr)
    \leq 0, 
\end{align*}
where $k \in \mathbb{R}$ and  a numerical entropy flux 
 is defined by  
\begin{align*}
 H^n_{j+1/2}=H^n_{j+1/2}\bigl(v^n_j,v^n_{j+1},k\bigr):=
  h\bigl(v^{n}_{j}\vee k,v^{n}_{j+1}\vee k\bigr)-h\bigl(v^{n}_{j}\wedge k,
 v^{n}_{j+1}\wedge k\bigr).
\end{align*}
By a  standard Lax-Wendroff-type argument  we  conclude that 
 as $\Delta \to 0$, 
 the piecewise constant functions assuming 
  the value $\smash{v_j^n}$ on $\smash{I_{j+1/2}\times [t_n, t_{n+1})}$ 
     converge to a limit function~$v$  that 
 for all non-negative test functions $\varphi 
 \in C_0^{\infty} ( \Pi_T)$ satisfies  
the entropy inequality
\begin{align*} \forall k \in \mathbb{R}: \quad 
\iint_{\Pi_T}  \Bigl\{  |v-k|\varphi_t+ \sgn (v-k)
 \bigl[ \Phi(v)-\Phi(k)-A(v_x)\bigr]\varphi_x\Bigr\} 
 \, \mathrm{d}x \,  \mathrm{d}t\geq 0. 
\end{align*}
However, since under the present assumptions $v(\cdot, t)$ is 
 a monotone smooth function, the entropy satisfaction property is 
not needed here.   

\end{remark}

\begin{lemma} \label{lemm4.4} 
The   solution  $\smash{ \{ u_{j+1/2}^n \}}$  
 generated by the $u$-scheme 
  \eqref{scheme100}, \eqref{scheme101} satisfies the 
   following inequality, where  the  
 constant $C_4$ 
is  independent of~$\Delta$: 
\begin{align*}
\TV \bigl( A(U^n) \bigr) = 
 \sum_{j\in\mathbb{Z}}\bigl| \Delta_+ A\bigl(u^{n}_{j-1/2} \bigr)\bigr|\leq C_4.
\end{align*}
\end{lemma}
\begin{proof}
 Using the marching formula  \eqref{marching} we can write
\begin{align*}
 \bigl| \Delta_+ A \bigl(u^{n}_{j-1/2} \bigr)\bigr|&\leq \fracd{1}{\lambda}  \bigl|v^{n+1}_{j}-v^{n}_{j}
  \bigr|+\bigl|\Delta_+ h \bigl(v^{n}_{j-1},v^{n}_{j}\bigr)\bigr| \\
  & \leq \fracd{1}{\lambda}   \bigl|v^{n+1}_{j}-v^{n}_{j}
  \bigr| + \bigl|  \bigl[ h \bigl(v^{n}_{j},v^{n}_{j+1}\bigr)-h\bigl(v^{n}_{j},v^{n}_{j}
 \bigr) \bigr] \theta \bigl( v^{n}_{j+1}-v^{n}_{j} \bigr)\bigr| \bigl|\Delta_+ v^{n}_{j}\bigr| 
 \\ & \quad  + \bigl|  \bigl[ h \bigl(v^{n}_{j},v^{n}_{j} \bigr)-h \bigl(v^{n}_{j-1},v^{n}_{j}\bigr)
  \bigr] \theta \bigl( v^{n}_{j}-v^{n}_{j-1} \bigr) \bigr|  \bigl| \Delta_- v^{n}_{j} \bigr|.
\end{align*}
Summing over $j\in\mathbb{Z}$ yields 
\begin{displaymath}
 \sum_{j\in\mathbb{Z}}\bigl| \Delta_+ A\bigl(u^{n}_{j-1/2} \bigr)\bigr|\leq \frac{1}{\lambda} 
  \sum_{j\in\mathbb{Z}}\bigl| v^{n+1}_{j}-v^{n}_{j} \bigr|+2\Vert \Phi'\Vert_{\infty}\sum_{j\in\mathbb{Z}}
  \bigl| \Delta_+ v^{n}_{j} \bigr|.
\end{displaymath}
The right-hand side is uniformly 
 bounded  due to  Lemma~\ref{LC_respect_t} and  
  Corollary~\ref{crol4.1}. 
\end{proof}

 Lemma~\ref{lemm4.4} does, in general, not permit to establish 
a uniform bound on the spatial total variation $\TV(U^n)$ of 
 the solution values $\smash{ \{ u_{j+1/2}^n \}}$ generated by the $u$-scheme. 
  This is possible only in the special case that $A(\cdot)$ is strictly 
   increasing as a function of~$u$, and $a(u)$ vanishes at most at 
   isolated values of~$u$. 
   
   We  now prove   that  $\TV(U^n)$ 
    is nevertheless uniformly bounded, but by a bound that depends 
    on the final time~$T$.   
Our analysis  
 will appeal to assumption \eqref{phiass}. 
  From \eqref{wa} and  \eqref{monotalways} we deduce 
   that if $\smash{\{v_j^n \}}$ is the numerical solution 
    produced by the $v$-scheme \eqref{marching}--\eqref{phipmdef}, 
 then at each time level 
    there exists a unique index~$k$ such that $\smash{v^n_k<v^*\leq v^n_{k+1}}$. 
The following lemma informs about the behavior of this index with each time iteration. 

\begin{lemma} \label{position_of_v^*} Assume that the 
 data $\smash{ \{v_j^n\}_{j \in \mathbb{Z}} }$ and  $\smash{ \{v_j^{n+1} \}_{j \in \mathbb{Z}} }$ 
 have been produced by the $v$-scheme  \eqref{marching}--\eqref{phipmdef} starting from the monotone data 
  $\smash{ \{ v_j^0 \}_{j \in \mathbb{Z}}}$ under the CFL condition  
   \eqref{cflcond}. Let   $k, \bar{k} \in \mathbb{Z}$ be 
  the uniquely defined indices that satisfy  
$\smash{v^{n}_{k}<v^*\leq v^{n}_{k+1}}$ and 
$\smash{v^{n+1}_{\bar{k}}<v^*\leq v^{n+1}_{\bar{k}+1}}$, respectively. 
 Then $\smash{\bar{k} \in \{ k-1,k,k+1\}}$.
 \end{lemma} 
\begin{proof}
 Since $v^{n}_{k}<v^*\leq v^{n}_{k+1}$ we analyze two cases: $v^{n}_{k}<v^*< v^{n}_{k+1}$ and $v^{n}_{k}<v^*=v^{n}_{k+1}$. In the first, 
 the monotonicity of the $v$-scheme 
   and \eqref{monotalways} imply that 
\begin{align*}
 v^{n+1}_{k-1}\leq v^{n}_{k}<v^*<v^{n}_{k+1}\leq v^{n+1}_{k+2}, 
\end{align*}
such that either 
$\smash{v^{n+1}_{k-1}<v^*\leq v^{n+1}_{k}}$, or 
 $\smash{v^{n+1}_{k}<v^*\leq v^{n+1}_{k+1}}$, or 
  $\smash{v^{n+1}_{k+1}<v^*< v^{n+1}_{k+2}}$,  
which means that  $\bar{k}=\left\lbrace k-1,k,k+1\right\rbrace$. In the second, 
 we find that 
\begin{align*}
 v^{n+1}_{k-1}\leq v^{n}_{k}<v^*=v^{n}_{k+1}\leq v^{n+1}_{k+2}, 
\end{align*}
so either 
$\smash{v^{n+1}_{k-1}<v^*= v^{n+1}_{k}}$, or $\smash{v^{n+1}_{k}<v^*\leq v^{n+1}_{k+1}}$, 
  or $\smash{v^{n+1}_{k+1}<v^*\leq  v^{n+1}_{k+2}}$.  We conclude 
   the proof by noting that 
  $v^{n+1}_{k+2}<v^*$ is  impossible due to the monotonicity of the 
   $v$-scheme   and \eqref{monotalways}.
  \end{proof}

The next lemma states the 
 announced bound on $\TV (U^n)$.

\begin{lemma}\label{u_is_TV_stable} Assume that  
 the CFL condition \eqref{cflcond} is satisfied. Then 
  there exist constants~$C_5$ and~$C_6$, 
   which are independent of~$\Delta$, such that the 
    solution values $\smash{U^n= \{u_{j+1/2}^n \}_{j \in \mathbb{Z}}}$ satisfy the 
     uniform total variation bound  
\begin{align}  \label{tvun} 
 \TV ( U^n)   = \sum_{j \in \mathbb{Z}} 
  \bigl| u_{j+1/2}^n - u_{j-1/2}^n \bigr|  \leq \bigl(C_5 + \TV (U^0) \bigr) \exp (C_6 T), 
   \quad n=1, \dots,   N. 
   \end{align}    
\end{lemma}
\begin{proof}
 From \eqref{scheme100} 
 we obtain 
\begin{align*}
\Delta_+ u^{n+1}_{j-1/2}&=\Delta_+ u^{n}_{j-1/2}- \mu \Delta_+ 
\Delta^2 h \bigl(v^{n}_{j-1},v^{n}_{j}\bigr)
 +\mu \Delta_+ \Delta^2  A \bigl(u^{n}_{j-1/2} \bigr).
\end{align*}
Let~$k$ be the index such that $v^n_k<v^*\leq v^n_{k+1} $, 
 and let us split $\mathbb{Z}$ into  the subsets 
 \begin{align} \label{usefulsets} \begin{split} 
 \mathcal{A} & := \mathcal{A}^n := \{ j\in\mathbb{Z} \, | \,  j\leq k-2\}, \\
 \mathcal{B}& := \mathcal{B}^n := \{ j\in 
  \mathbb{Z} \, | \, k-2< j\leq k+2\}, \\ 
  \mathcal{C}&:=\mathcal{C}^n := \{ j\in  \mathbb{Z} \, | \, 
   k+2<j \}.
  \end{split} \end{align}   
Let $\smash{  w^{n}_{j}:=\Delta_+ u^{n}_{j-1/2}}$ and 
 $\smash{a^n_j:=
\Delta_+ A(u^{n}_{j-1/2}) \theta (  \Delta_+ u^{n}_{j-1/2} )}$.
For $j\in \mathcal{A}$, we obtain 
\begin{align} \label{jinaident} 
  w^{n+1}_{j}= w^{n}_{j}- \mu \Delta_- \Delta^2 \Phi  \bigl(v_j^n) 
  + \mu \Delta^2 \bigl( a_j^n w_j^n \bigr).  
\end{align}
 Using a Taylor expansion about $v^n_j$ we find that there exist 
numbers $\smash{\alpha^n_{j}\in[v^n_j,v^n_{j+1}]}$ and $\smash{\beta^n_{j}\in[v^n_{j-1},v^n_{j}]}$
 such that 
\begin{align*}
\Delta^2 \Phi \bigl(v^{n}_{j}\bigr)=\Phi'(v^{n}_{j})w^{n}_{j}\Delta x+\frac{1}{2} \Phi''\bigl(\alpha^n_{j}\bigr)
\bigl(\Delta_+ v^{n}_{j}\bigr)^2+ \frac{1}{2} \Phi''\bigl(\beta^n_{j}\bigr) \bigl(\Delta_- v^{n}_{j} \bigr)^2.
\end{align*}
Substituting this into \eqref{jinaident} we obtain 
\begin{align*}
 w^{n+1}_{j}&=w^{n}_{j}-\lambda \Delta_- \bigl( \Phi'\bigl(v^n_{j}\bigr)w^n_{j}\bigr) 
+  \mu  \Delta^2 \bigl( a^{n}_{j}w^{n}_{j} \bigr) -\frac{\mu}{2}  \Delta_- \bigl( \Phi'' \bigl( \alpha^n_{j}\bigr)\bigl(\Delta_+ v^{n}_{j} \bigr)^2\bigr) \\
& \quad -\frac{\mu}{2}  \Delta_- \bigl( \Phi'' \bigl( \beta^n_{j}\bigr)\bigl(\Delta_- v^{n}_{j}\bigr)^2\bigr)\\
& =w^{n}_{j}-\lambda \Delta_- \bigl( \Phi'\bigl(v^n_{j}\bigr)w^n_{j}\bigr) 
+  \mu  \Delta^2 \bigl( a^{n}_{j}w^{n}_{j} \bigr) \\
& \quad  -\frac{\mu}{2}\Bigl( \Delta_- \Phi''\bigl(\alpha^n_{j}\bigr)\bigl(\Delta_+ v^{n}_{j}\bigr)^2+ \Phi''\bigl(\alpha^n_{j-1}\bigr)\bigl(v^{n}_{j+1}-v^{n}_{j-1}\bigr)w^n_{j}\Delta x \\
& \qquad 
  + \Delta_- \Phi''\bigl(\beta^n_{j}\bigr) \bigl(\Delta_- v^{n}_{j}\bigr)^2+ 
  \Phi''\bigl(\beta^n_{j-1}\bigr) \bigl(v^{n}_{j}-v^{n}_{j-2}\bigr)w^n_{j-1}\Delta x\Bigr)
 \\&=w^{n}_{j}\bigl[ 1-\lambda\Phi'\bigl(v^n_{j}\bigr)-2\mu a^{n}_{j}\bigr]
 +w^{n}_{j-1}\bigl[\mu a^{n}_{j-1}+\lambda \Phi'\bigl(v^n_{j-1}\bigr)\bigr]+\mu w^{n}_{j+1}
  a^{n}_{j+1} \\
&\quad +\mathcal{O}(\Delta t) \bigl(w^{n}_{j-1}+w^{n}_{j}+ \Delta_+ v^{n}_{j}   +  \Delta_-  v^{n}_{j} \bigr).
\end{align*}
In an analogous way, we find for $j\in \mathcal{C}$
\begin{align*}
 w^{n+1}_{j}&=w^{n}_{j}\bigl[ 1+\lambda\Phi'\bigl(v^n_{j}\bigr)-2\mu a^{n}_{j}\bigr]
 +w^{n}_{j+1}\bigl[\mu a^{n}_{j+1}-\lambda \Phi'\bigl(v^n_{j+1}\bigr)\bigr]+\mu 
 w^{n}_{j-1} a^{n}_{j-1} \\
&\quad +\mathcal{O}(\Delta t) \bigl(w^{n}_{j}+w^{n}_{j+1}
 +  \Delta_+ v^{n}_{j}  + \Delta_- v^{n}_{j}\bigr).
\end{align*}
Now we deal with $j\in \mathcal{B}$. For $j=k-1$, using that $v^*$ is a maximum of~$\Phi$ and following analogous steps as before,  we get 
 \begin{align*}
 w^{n+1}_{k-1}&=w^{n}_{k-1}-\mu \bigl(\Phi(v^{n}_{k+1})-\Phi(v^{*})+\Delta_- \Delta^2 \Phi  \bigl(v_{k-1}^n\bigr)\bigr) 
 + \mu \Delta^2 \bigl( a_{k-1}^n w_{k-1}^n \bigr)\\
 & = w^{n}_{k-1}-\mu \bigl(\Phi'(\xi)\bigl(v^{n}_{k+1}-v^{*}\bigr)+\Delta_- \Delta^2 \Phi  \bigl(v_{k-1}^n\bigr)\bigr)+ \mu \Delta^2 \bigl( a_{k-1}^n w_{k-1}^n\bigr)\\
 &= w^{n}_{k-1}-\mu \bigl( \bigl(\Phi'(\xi)-\Phi'(v^*)\bigr)
  \bigl(v^{n}_{k+1}-v^{*}\bigr)+\Delta_- \Delta^2 \Phi  \bigl(v_{k-1}^n\bigr)\bigr) \\ & \quad 
  + \mu \Delta^2 \bigl( a_{k-1}^n w_{k-1}^n\bigr)\\
 &= w^{n}_{k-1}\bigl[ 1-\lambda\Phi'\bigl(v^n_{k-1}\bigr)-2\mu a^{n}_{k-1}\bigr]
 +w^{n}_{k-2}\bigl[\mu a^{n}_{k-2}+\lambda \Phi'\bigl(v^n_{k-2}\bigr)\bigr]+\mu w^{n}_{k}
  a^{n}_{k} \\
  &\quad +\mathcal{O}(\Delta t)\bigl( 1+ w^{n}_{k-2}+w^{n}_{k-1}+
  \Delta_+ v^{n}_{k-1}   + \Delta_-  v^{n}_{k-1}  \bigr). 
  \end{align*}
For $j=k$, using that $\Phi'(v^*)=0$ we compute
\begin{align*}
w^{n+1}_{k}&=w^{n}_{k} -\mu \bigl[\Phi \bigl(v^{n}_{k+2} \bigr)-2\Phi(v^{n}_{k+1})+\Phi(v^{n}_{k}) -\left\lbrace \Phi(v^{n}_{k})-2\Phi(v^{n}_{k-1})+\Phi\bigl(v^{n}_{k-2}\bigr)\right\rbrace\bigr]\\
&\quad -\mu\bigl[ \Phi(v^{n}_{k-1})-\Phi(v^{n}_{k})+2\bigl(\Phi(v^{*})-\Phi(v^{n}_{k})\bigr)+\Phi(v^{*})-\Phi(v^{n}_{k+1})\bigr] 
\\ &\quad 
 + \mu \Delta^2 \bigl( a_{k}^n w_{k}^n \bigr)\\
&=w^{n}_{k} -\mu \bigl[\Delta_+\Delta^2\Phi(v^{n}_{k})+\Delta_-\Delta^2\Phi(v^{n}_{k})\bigr]+ \mu \Delta^2 \bigl( a_{k}^n w_{k}^n \bigr)\\
&\quad -\mu\bigl[ \Phi \bigl(v^{n}_{k-1}\bigr)-\Phi(v^{n}_{k}) 
+2\bigl(\Phi(v^{*})-\Phi(v^{n}_{k})\bigr)+\Phi(v^{*})-\Phi\bigl(v^{n}_{k+1}\bigr)\bigr]\\
&= w^{n}_{k}\bigl( 1-2\mu a^{n}_{k}\bigr)
 +w^{n}_{k-1}\bigl[\mu a^{n}_{k-1}+\lambda \Phi'\bigl(v^n_{k-1}\bigr)\bigr]+w^{n}_{k+1}\bigl[\mu a^{n}_{k+1}-\lambda \Phi'\bigl(v^n_{k+1}\bigr)\bigr]\\
 &\quad +\mathcal{O}(\Delta t) \bigl(1+ 
 w^{n}_{k-1}+w^{n}_{k}+w^{n}_{k+1} + \Delta_+ v^{n}_{k}   +  \Delta_-  v^{n}_{k}  \bigr).
\end{align*}
For $j=k+1$ and $j=k+2$,  the following 
steps are analogous to the  previous  cases. Using  that $\Phi'(v^*)=0$ we obtain
\begin{align*}
w^{n+1}_{k+1}&=w^{n}_{k+1}-\mu\bigl[ \Delta_+\Delta^2\Phi(v^{n}_{k+1})+3(\Phi(v^{n}_{k})-\Phi(v^{*}))+\Phi(v^{n}_{k})-\Phi(v^{n}_{k-1})\bigr]\\
&\quad + \mu \Delta^2 \bigl( a_{k+1}^n w_{k+1}^n \bigr)\\
&=w^{n}_{k+1}\bigl[ 1+\lambda\Phi'\bigl(v^n_{k+1}\bigr)-2\mu a^{n}_{k+1}\bigr]
 +w^{n}_{k}\mu a^{n}_{k}+w^{n}_{k+2}\bigl[\mu a^{n}_{k+2}-\lambda \Phi'\bigl(v^n_{k+2}\bigr)\bigr]\\
 &\quad +\mathcal{O}(\Delta t) \bigl(1+w^{n}_{k+1}+w^{n}_{k+2}
  + \Delta_+ v^{n}_{k+1}   + \Delta_-  v^{n}_{k+1}  \bigr), \\
   w^{n+1}_{k+2}&=w^{n}_{k+2}-\mu \bigl[\Delta_+\Delta^2\Phi(v^{n}_{k+2})+\Phi(v^{n}_{k})-\Phi(v^*)\bigr]+ \mu \Delta^2 \bigl( a_{k+2}^n w_{k+2}^n \bigr)\\
&=w^{n}_{k+2}\bigl[ 1+\lambda\Phi'\bigl(v^n_{k+2}\bigr)-2\mu a^{n}_{k+2}\bigr]
 +w^{n}_{k+3}\bigl[\mu a^{n}_{k+3}-\lambda \Phi'\bigl(v^n_{k+3}\bigr)\bigr]+\mu 
 w^{n}_{k+1} a^{n}_{k+1} \\
&\quad +\mathcal{O}(\Delta t) \bigl(1+ w^{n}_{k+2}+w^{n}_{k+3} 
+ \Delta_+ v^{n}_{k+2}   +  \Delta_- v^{n}_{k+2}\bigr).
\end{align*}
Finally, summing over $j$ we find that there exist constants~$C_6$ and~$C_7$
 such that   
\begin{align*}
 \sum_{j\in \mathbb{Z}}\bigl|w^{n+1}_{j}\bigr|\leq&\sum_{j\in \mathbb{Z}}\left|w^{n}_{j}\right|(1+
  C_6\Delta t)+ C_7 \Delta t,
\end{align*}
which implies that  
\begin{align*}
\sum_{j\in \mathbb{Z}}\left|w^{n+1}_{j}\right|\leq &\sum_{j\in \mathbb{Z}
}\left|w^{0}_{j}\right|\exp(C_6 T)+\frac{C_7}{C_6}\exp(C_6 T),  
\end{align*}
which proves \eqref{tvun}. 
\end{proof}

The next lemma states $L^1$ H\"{o}lder continuity with respect to the variable $t$ of the solution generated  by \eqref{scheme100}.

\begin{lemma}\label{L1_holder_continuity}
The   solution  $\smash{ \{ u_{j+1/2}^n \}}$  
 generated by the $u$-scheme 
  \eqref{scheme100}, \eqref{scheme101} satisfies the 
   following inequality, where  the  
 constant $C_8$ 
is  independent of~$\Delta$:
\begin{align} \label{c8ineq} 
 \sum_{j\in\mathbb{Z}}\bigl|u^{m}_{j+1/2}-u^{n}_{j+1/2}\bigr|\Delta x\leq&\,C_8\sqrt{\Delta t(m-n)} \quad \text{\em for $m>n$, $m,n\in\mathbb{N}_0$}.
\end{align}
\end{lemma}
\begin{proof}
We first establish weak Lipschitz continuity in the time variable. To this end, let $\phi(x)$ be a test function and $\phi_j:=\phi(j\Delta x)$. Multiplying equation \eqref{scheme100} by $\phi_j\Delta x$, summing over $n$ and $j$ and applying  a summation by parts, we get
\begin{align*}
 \Biggl|\Delta x\sum_{j\in\mathbb{Z}}\phi_j\bigl(u^{n+1}_{j+1/2}-u^{n}_{j+1/2}\bigr)\Biggr|&= \Biggl|\Delta t\sum_{j\in\mathbb{Z}}G^n_j\left(\phi_{j}-\phi_{j-1}\right)\Biggr|\\
&\quad +\Biggl|\lambda\sum_{j\in\mathbb{Z}}\left(\phi_{j}-\phi_{j-1}\right)\bigl(A\bigl(u^{n}_{j+1/2}\bigr)-A\bigl(u^{n}_{j-1/2}\bigr)\bigr)\Biggr|.
\end{align*}
Using Lemma \ref{lemm4.4} and the fact that $\phi$ is smooth we obtain
\begin{align*}
 \Biggl|\Delta x\sum_{j\in\mathbb{Z}}\phi_j\bigl(u^{n+1}_{j+1/2}-u^{n}_{j+1/2}\bigr)\Biggr|\leq C\Vert \phi'\Vert\Delta t, 
\end{align*}
where $C$ is independent of $\Delta$ and $\phi$. Consequently, for $m>n$ the following weak continuity result holds: 
\begin{align*}
 \Biggl|\Delta x\sum_{j\in\mathbb{Z}}\phi_j\bigl(u^{m}_{j+1/2}-u^{n}_{j+1/2}\bigr)\Biggr|\leq& C\Vert \phi'\Vert\Delta t(m-n).
\end{align*}
Since $\smash{E_j:=u^{m}_{j+1/2}-u^{n}_{j+1/2}}$ has bounded variation on $\mathbb{R}$, we arrive 
at the inequality \eqref{c8ineq}  by 
 proceeding as in \cite[Lemma 3.6]{ek00}. 
  \end{proof}

Now, following ideas of \cite{kr01} we prove an $L^2$ estimate for $A'(\cdot)_x$.

\begin{lemma}\label{L^2_A_x}
 The   solution  $\smash{ \{ u_{j+1/2}^n \}}$  
 generated by the $u$-scheme 
  \eqref{scheme100}, \eqref{scheme101} satisfies the 
   following inequality, where  the  
 constant $C_{9}$ 
is  independent of~$\Delta$:
\begin{align}
 \sum_{n=1}^N \sum_{j\in\mathbb{Z}}\left(\frac{\Delta_-  A(u^{n}_{j+1/2}) }{\Delta x}\right)^2\Delta t\Delta x\leq C_{9}.
\end{align}
\end{lemma}
\begin{proof}
 Multiplying  \eqref{scheme100}, by $\smash{u^{n}_{j+1/2}\Delta x}$, summing the result over~$n=0,\dots, N-1$ and~$j\in\mathbb{Z}$, and using 
   summations by parts we get
\begin{align*}
& \lambda\sum_{n=0}^{N-1} \sum_{j\in\mathbb{Z}}\bigl(\Delta_- A\bigl(u^{n}_{j+1/2}\bigr) \bigr)\bigl(\Delta_- u^{n}_{j+1/2} \bigr)\\
& =\Delta t\sum_{n=0}^{N-1} \sum_{j\in\mathbb{Z}}G^n_{j}\bigl(\Delta_- u^{n}_{j+1/2} \bigr)
- \frac{\Delta x}{2}\sum_{n=0}^{N-1} \sum_{j\in\mathbb{Z}}\bigl(\bigl(u^{n+1}_{j+1/2}\bigr)^2-\bigl(u^{n}_{j+1/2}\bigr)^2\bigr) \\
& \quad + \frac{\Delta x}{2}\sum_{n=0}^{N-1} \sum_{j\in\mathbb{Z}}\bigl(\Delta_- u^{n+1}_{j+1/2} \bigr)^2, 
\end{align*}
where we used that 
\begin{align*} 
\bigl(u^{n+1}_{j+1/2}-u^{n}_{j+1/2} \bigr)u^{n}_{j+1/2}=\frac{1}{2}\Bigl[\bigl(u^{n+1}_{j+1/2}\bigr)^2-\bigl(u^{n}_{j+1/2}\bigr)^2-\bigl(u^{n+1}_{j+1/2}-u^{n}_{j+1/2}\bigr)^2\Bigr] .
\end{align*}  
In light of  Lemma~\ref{L_infinity_u}, we can also write
\begin{align*} 
\bigl(  \Delta_- A\bigl(u^{n}_{j+1/2}\bigr) \bigr) \bigl(\Delta_- 
 u^{n}_{j+1/2} \bigr)\geq \frac{1}{a^*} \bigl(\Delta_- A\bigl(u^{n}_{j+1/2}\bigr)\bigr)^2,
  \quad a^* := \max_{u} a(u), 
\end{align*} 
since $a(u)\geq 0$. Using this observation, we find that
\begin{align}\label{eq_ast} \begin{split} 
 \frac{\lambda}{a^*} \sum_{n=0}^{N-1}\sum_{j\in\mathbb{Z}}\bigl(\Delta_- A\bigl(u^{n}_{j+1/2}\bigr)\bigr)^2 
& \leq\Delta t\sum_{n=0}^{N-1}\sum_{j\in\mathbb{Z}}G^n_{j}\bigl(\Delta_- u^{n}_{j+1/2} \bigr)
+ \frac{\Delta x}{2}\sum_{j\in\mathbb{Z}}\bigl(u^{0}_{j+1/2}\bigr)^2  \\
&\quad+ \frac{\Delta x}{2}\sum_{n=0}^{N-1} \sum_{j\in\mathbb{Z}}\bigl(u^{n+1}_{j+1/2}-u^{n}_{j+1/2}\bigr)^2.
\end{split} 
\end{align}
On the other hand, from \eqref{scheme100} and the inequality $ (a+b)^2\leq 2a^2+2b^2$ we obtain
\begin{align*}
 \frac{1}{2}\bigl(u^{n+1}_{j+1/2}-u^{n}_{j+1/2}\bigr)^2
  & \leq \lambda^2\bigl(\Delta_+ G^n_j\bigr)^2 +2\mu^2\Bigl( \bigl( \Delta_+ A \bigl(u^{n}_{j+1/2}\bigr)
   \bigr)^2+\bigl( \Delta_- A\bigl(u^{n}_{j+1/2}\bigr) \bigr)^2\Bigr).
\end{align*}
Multiplying the last inequality by $\Delta x$ and summing the result over~$n$ and~$j$  yields 
\begin{align*}
 \frac{\Delta x}{2}\sum_{n=0}^{N-1}\sum_{j\in\mathbb{Z}}\bigl(u^{n+1}_{j+1/2}-u^{n}_{j+1/2}\bigr)^2&\leq \frac{\Delta t^2}{\Delta x}\sum_{n=0}^{N-1} \sum_{j\in\mathbb{Z}}\bigl(\Delta_+ G^n_j\bigr)^2
 \\ &\quad 
 +4\mu^2\Delta x\sum_{n=0}^{N-1}\sum_{j\in\mathbb{Z}}\bigl( \Delta_- A\bigl(u^{n}_{j+1/2}\bigr) \bigr)^2.
\end{align*}
In what follows, we assume that the following strengthened CFL condition is satisfied for a 
constant $\eps >0$:  
\begin{align}
\textrm{CFL}_\eps:= 2\lambda  \max_{u \in \mathbb{R}}       \bigl| \Phi'(u) \bigr|  
 +  4\mu \max_{u \in \mathbb{R}} a(u) \leq 1-\eps.\label{CFL_eps} 
\end{align}
The new CFL condition implies in particular that
\begin{align*} 
 4\mu^2\Delta x=4\mu \frac{\Delta t}{\Delta x}\leq \frac{\Delta t(1-\eps)}{\Delta x \, a^*}, 
 \end{align*} 
and therefore
\begin{align} \label{eq_cat} \begin{split} 
 &\frac{\Delta x}{2}\sum_{n=0}^{N-1}\sum_{j\in\mathbb{Z}}\bigl(u^{n+1}_{j+1/2}-u^{n}_{j+1/2}\bigr)^2\\
 & \leq \frac{\Delta t^2}{\Delta x}\sum_{n=0}^{N-1} \sum_{j\in\mathbb{Z}}\bigl(\Delta_+ 
 G^n_j\bigr)^2   +\frac{\Delta t(1-\eps)}{\Delta x \, a^*}\sum_{n=0}^{N-1}\sum_{j\in\mathbb{Z}}\bigl( \Delta_- A\bigl(u^{n}_{j+1/2}\bigr) \bigr)^2.
\end{split} 
\end{align}
Summing \eqref{eq_ast} and \eqref{eq_cat} yields 
\begin{align*}
& \frac{\eps\lambda}{a^*} \sum_{n=0}^{N-1} \sum_{j\in\mathbb{Z}}\bigl(\Delta_- A \bigl(u^{n}_{j+1/2}\bigr)\bigr)^2  \\
&\leq\Delta t\sum_{n=0}^{N-1} \sum_{j\in\mathbb{Z}}G^n_{j}\bigl(  \Delta_- u^{n}_{j+1/2} \bigr)
+ \frac{\Delta x}{2}\sum_{j\in\mathbb{Z}}\bigl(u^{0}_{j+1/2}\bigr)^2
 +\frac{\Delta t^2}{\Delta x}\sum_{n=0}^{N-1} \sum_{j\in\mathbb{Z}}\bigl(\Delta_- 
G^n_{j+1}\bigr)^2  \leq C,
\end{align*}
where we used Lemma \ref{u_is_TV_stable}, the bound over $G^n_j$ and the fact that $\Delta t=\mathcal{O}(\Delta x^2)$. 
\end{proof}

     With the help of Lemma~\ref{L^2_A_x}  we can prove 
     \begin{lemma} \label{lemma4.10} Under the assumptions of Lemma~\ref{L^2_A_x}
      there exists a constant $C_{10}$ which is independent of~$\Delta$ 
       such that 
        \begin{align} 
          \sum_{j\in\mathbb{Z}}\bigl| A \bigl(u_{j+1/2}^m \bigr) -  A \bigl(u_{j+1/2}^n \bigr) \bigr|^2\Delta x 
           \leq C_{10}  (m-n) \Delta t \quad \text{\em  for $m>n$.}  
        \end{align} 
     \end{lemma} 
\begin{proof}
 Using Lemma \ref{L_infinity_u}, the fact that $A'(u)\geq 0$ and \eqref{scheme100} we get
\begin{align} \label{deffab} \begin{split} 
 &\sum_{j\in\mathbb{Z}}\bigl( A \bigl(u^{m}_{j+1/2}\bigr)-A\bigl(u^{n}_{j+1/2}\bigr)\bigr)^2\Delta x\\
&\leq a^* \sum_{j\in\mathbb{Z}}\bigl( A \bigl(u^{m}_{j+1/2}\bigr)-A \bigl(u^{n}_{j+1/2}\bigr)\bigr)\bigl( u^{m}_{j+1/2}-u^{n}_{j+1/2}\bigr) \Delta x=: \mathcal{A} + \mathcal{B}, 
\end{split}\end{align}   
where we define 
\begin{align*} 
\mathcal{A} & :=-\Delta t \, a^* \sum_{j\in\mathbb{Z}}\bigl( A \bigl(u^{m}_{j+1/2}\bigr)-
A\bigl(u^{n}_{j+1/2}\bigr)\bigr) \sum^{m-1}_{l=n} \Delta_+G^l_j,\\
\mathcal{B} & := \lambda\, a^* \sum_{j\in\mathbb{Z}}
\bigl( A \bigl(u^{m}_{j+1/2}\bigr)-A \bigl(u^{n}_{j+1/2}\bigr)\bigr) 
\sum^{m-1}_{l=n} \Delta^2A \bigl(u^{l}_{j+1/2}\bigr).  
\end{align*}
Summing by parts, using the Cauchy-Schwarz inequality and Lemma \ref{L^2_A_x}  we obtain
\begin{align*}
 \mathcal{A}&= \Delta t \, a^* \sum_{j\in\mathbb{Z}}\sum^{m-1}_{l=n}G^l_j\Delta_+
 \bigl( A \bigl(u^{m}_{j+1/2}\bigr)-A \bigl(u^{n}_{j+1/2}\bigr)\bigr) \\
&= \Delta t\, a^* \sum_{j\in\mathbb{Z}}\sum^{m-1}_{l=n}G^l_j
\bigl( \Delta_- A \bigl(u^{m}_{j+1/2}\bigr) -\Delta_- A \bigl(u^{n}_{j+1/2}\bigr) \bigr)\\
&\leq \frac{\Delta t}{2} a^* \sum_{j\in\mathbb{Z}}\sum^{m-1}_{l=n}\left|G^l_j\right|\left[\left(\frac{\Delta_- A(u^{m}_{j+1/2})}{\Delta x}\right)^2+ \left(\frac{\Delta_- A(u^{n}_{j+1/2})}{\Delta x}\right)^2\right]\Delta x\\
&\quad + \frac{\Delta t}{2} a^* \sum_{j\in\mathbb{Z}}\sum^{m-1}_{l=n}\left|G^l_j\right|\Delta x
 = \mathcal{O}\bigl((m-n)\Delta t\bigr).
\end{align*}
Proceeding in the same way for $\mathcal{B}$ yields  
\begin{align*}
 \mathcal{B}&= -\lambda a^* \sum_{j\in\mathbb{Z}} \biggl\{ \bigl[ A \bigl(u^{m}_{j+1/2}\bigr)-
 A\bigl(u^{n}_{j+1/2}\bigr)-\bigl(A \bigl(u^{m}_{j-1/2}\bigr)-A\bigl(u^{n}_{j-1/2}\bigr)\bigr)\bigr]\\
&\quad \times\sum^{m-1}_{l=n}\Delta_- A \bigl(u^{l}_{j+1/2}\bigr)\biggr\}  \\ 
&=-\lambda a^*\sum_{j\in\mathbb{Z}}\biggl\{ \bigl( \Delta_- A \bigl(u^{m}_{j+1/2}\bigr)
-\Delta_- A \bigl(u^{n}_{j+1/2}\bigr) \bigr) 
 \sum^{m-1}_{l=n}\Delta_- A \bigl(u^{l}_{j+1/2}\bigr) \biggr\} \\
&=-\lambda a^* \sum_{j\in\mathbb{Z}}\sum^{m-1}_{l=n}\left(\Delta_- A\bigl(u^{m}_{j+1/2}\bigr)
 \cdot \Delta_- A\bigl(u^{l}_{j+1/2}\bigr) 
-\Delta_- A\bigl(u^{n}_{j+1/2}\bigr) \cdot \Delta_- A\bigl(u^{l}_{j+1/2}\bigr)
\right)  \\
&\leq 2(m-n)\Delta t \, a^* \sum_{j\in\mathbb{Z}}\left(\frac{\Delta_- A(u^{n}_{j+1/2})}{\Delta x}\right)^2\Delta x =\mathcal{O}\bigl((m-n)\Delta t\bigr).
\end{align*}
Inserting into \eqref{deffab} that $\mathcal{A}, \mathcal{B} =\mathcal{O}((m-n)\Delta t)$ 
concludes the proof. 
\end{proof}

     Let us now denote by $u^{\Delta}$ 
      the  piecewise constant function 
       \begin{align} \label{pwcf}  
        u^{\Delta} (x, t) := \sum_{n=0}^{N-1} \sum_{j \in \mathbb{Z}} 
         \chi_{jn} (x,t) u_{j+1/2}^n, 
         \end{align} 
          where $\chi_{jn}$ denotes the characteristic function 
           of $I_j \times [t_n, t_{n+1})$,  
        and let us denote by $\smash{v^{\Delta}}$ its primitive. 
       From the $L^{\infty}$ bound 
         (Lemma~\ref{L_infinity_u}), the uniform bound on the total variation in space 
          (Lemma~\ref{u_is_TV_stable}) and the  $L^1$ H\"{o}lder continuity  in time 
          result (Lemma~\ref{L1_holder_continuity}) we infer that there is  a constant $C$ 
          such that 
           \begin{align}\label{compactness_u^Delta}
            \| u^{\Delta} \|_{L^{\infty} (\Pi_T)} + \|u^{\Delta}\|_{L^1(\Pi_T)} 
             \leq C;  \quad \bigl| u^{\Delta} (\cdot, t) \bigr|_{BV(\mathbb{R})} \leq 
              C \quad \text{for all $t \in (0, T]$}  
              \end{align} 
              uniformly as $\Delta x, \Delta t \downarrow 0$, 
              while Lemmas~\ref{L^2_A_x} and~\ref{lemma4.10} imply that there are 
                constants $C_{11}$ and $C_{12}$ independent of $\Delta$ such that 
	\begin{align}\label{compactness_A}\begin{split}
	 &\left\|A(u^\Delta(\cdot+y,\cdot))-A(u^\Delta(\cdot,\cdot))\right\|_{L^{2}(\Pi_{T})}\leq C_{11}(\vert y\vert +\Delta x),\\
	&\left\|A(u^\Delta(\cdot,\cdot+\tau))-A(u^\Delta(\cdot,\cdot))\right\|_{L^{2}(\Pi_{T})}\leq C_{12} \sqrt{ \tau +\Delta t}. 
	\end{split}
	\end{align}

\subsection{Convergence to the entropy solution} \label{subsec:convent} 

\begin{theorem} \label{th4.1} Assume that $\Delta x$ and $\Delta t$ satisfy the CFL condition 
\eqref{CFL_eps}, and that $u_0$ is compactly supported and satisfies 
 \eqref{initial_data_stability}.  Then the piecewise constant  solutions $\smash{u^{\Delta}}$ 
  generated by 
  the $u$-scheme  \eqref{scheme100}, \eqref{scheme101} converge in the strong topology of 
   $L^1 (\Pi_T)$ to an entropy solution of  \eqref{goveq}, \eqref{initcond}
    (in the sense of Definition~\ref{def:ent}). 
    \end{theorem} 

\begin{proof}  
%
%
%
Since $u^{\Delta} \in L^{\infty} (\Pi_T)  \cap 
 L^{\infty}(0, T; BV (\mathbb{R})) \cap C^{1/2}(0,T; L^1 (\mathbb{R}))$, 
  we deduce from \eqref{compactness_u^Delta} 
  that there exists a sequence $\smash{\{\Delta_i \}_{i \in \mathbb{N}}}$ 
with $\Delta_i \downarrow 0$ for $i \to \infty$ and a function 
 $u \in   L^{\infty}(\Pi_T) \cap L^1(\Pi_T) \cap L^{\infty} (0, T; BV(\mathbb{R}))$ such that 
  $\smash{u^{\Delta} \to u}$ a.e. on $\Pi_T$. Moreover, in light of 
   \eqref{compactness_A} we have 
    $\smash{A(u^{\Delta} ) \to A(u)}$ strongly on $L^2_{\textrm{loc}}(\Pi_T)$, and 
    we have that $A(u) \in L^2(0,T;H^1(\mathbb{R}))$. Lemma~\ref{L1_holder_continuity}
     ensures that~$u$ satisfies the initial condition \eqref{esslimcond}.   
    It remains 
   to prove that~$u$ satisfies the entropy inequality \eqref{entrineq}. 
    To this end, we show that the $u$-scheme satisfies a discrete entropy inequality, 
     and then apply a standard Lax-Wendroff-type argument.  
   From \eqref{scheme101} we infer that 
\begin{align*} 
 \Delta_+ G_j^n = \Delta_+ 
  \bigl[ u_{j-1/2}^n \Phi_+' \bigl( \alpha_{j-1/2}^n \bigr) 
   + u_{j+1/2}^n \Phi_-' \bigl( \beta_{j+1/2}^n \bigr) \bigr], 
   \end{align*} 
where $\smash{ \alpha_{j-1/2}^n, \beta_{j-1/2}^n \in [v_{j-1}^n, v_j^n] }$ 
satisfy 
\begin{align} 
 \Phi_+' \bigl( \alpha_{j-1/2}^n \bigr) &= \theta \bigl( 
  \Delta_+ v_{j-1}^n \bigr) 
   \int_{v_{j-1}^n}^{v_j^n} 
    \Phi_+' (s) \, \mathrm{d}s,  \label{scheme105} \\
 \Phi_-' \bigl( \beta_{j-1/2}^n \bigr) & = \theta \bigl( 
  \Delta_+ v_{j-1}^n \bigr) 
   \int_{v_{j-1}^n}^{v_j^n} 
    \Phi_-' (s) \, \mathrm{d}s. \label{scheme106} 
    \end{align}      
    Consequently, defining the function 
\begin{align*} 
 \mathcal{G}_{j+1/2}^n 
  (u,v,w):= &u \lambda  \Phi_+' \bigl( \alpha_{j-1/2}^n \bigr)
   + v \bigl[ 1 - \lambda \bigl(  \Phi_+' \bigl( \alpha_{j+1/2}^n \bigr)
    -  \Phi_-' \bigl( \beta_{j+1/2}^n \bigr) \bigr) \bigr] 
     \\& + w \bigl( - \lambda  \Phi_-' \bigl( \beta_{j+3/2}^n \bigr) \bigr)
      + \mu \bigl( A(u)- 2 A(v) + A(w) \bigr),
     \end{align*} 
     we can rewrite the scheme \eqref{scheme100} as 
 \begin{align*} 
  u_{j+1/2}^n =  \mathcal{G}_{j+1/2}^n \bigl( 
   u_{j-1/2}^n, u_{j+1/2}^n, u_{j+3/2}^n \bigr).  
 \end{align*}      
 Note that under the CFL condition, 
  $\smash{\mathcal{G}_{j+1/2}^n}$ is a monotone function of each of its 
  arguments for all $j \in \mathbb{Z}$ and $n =0,\dots, N-1$, and that
  \begin{align*} \forall k \in \mathbb{R}: \quad 
    \mathcal{G}_{j+1/2}^n (k,k,k) =k - \lambda k 
    \bigl[ \Delta_+  \Phi_+' \bigl( \alpha_{j-1/2}^n \bigr) 
     +\Delta_+   \Phi_-' \bigl( \beta_{j+1/2}^n \bigr) \bigr].  
     \end{align*}
     The  quantity
        $\smash{\bar{u}_{j+1/2}^{n+1} := u_{j+1/2}^n - \lambda
       \Delta_+ G_j^n +\mu \Delta^2 A (u_{j+1/2}^n)}$ 
       satisfies for all $k \in \mathbb{R}$  
       \begin{align*} 
       \bar{u}_{j+1/2}^{n+1} -k &=  \mathcal{G}_{j+1/2}^n \bigl( 
   u_{j-1/2}^n, u_{j+1/2}^n, u_{j+3/2}^n \bigr)- 
         	  \mathcal{G}_{j+1/2}^n (k,k,k)\\&\quad   - \lambda k  \bigl[ \Delta_+  \Phi_+' \bigl( \alpha_{j-1/2}^n \bigr) 
     +\Delta_+   \Phi_-' \bigl( \beta_{j+1/2}^n \bigr) \bigr],
     \end{align*} 
     and since $\smash{\mathcal{G}_{j+1/2}^n}$ is a monotone function of each of its 
  arguments, we get 
  \begin{align*} 
  & \bigl| \bar{u}_{j+1/2}^{n+1} -k + 
  \lambda k  \bigl[ \Delta_+  \Phi_+' \bigl( \alpha_{j-1/2}^n \bigr) 
     +\Delta_+   \Phi_-' \bigl( \beta_{j+1/2}^n \bigr) \bigr] \bigr| \\
     & = \bigl| \mathcal{G}_{j+1/2}^n \bigl( 
   u_{j-1/2}^n, u_{j+1/2}^n, u_{j+3/2}^n \bigr)- 
         	  \mathcal{G}_{j+1/2}^n (k,k,k) \bigr| \\
	  & = \mathcal{G}_{j+1/2}^n \bigl( 
   u_{j-1/2}^n, u_{j+1/2}^n, u_{j+3/2}^n \bigr) \vee  
         	  \mathcal{G}_{j+1/2}^n (k,k,k) \\& \quad - 
	  \mathcal{G}_{j+1/2}^n \bigl( 
   u_{j-1/2}^n, u_{j+1/2}^n, u_{j+3/2}^n \bigr) \wedge  
         	  \mathcal{G}_{j+1/2}^n (k,k,k) \\
	  & \leq  \mathcal{G}_{j+1/2}^n \bigl( 
   u_{j-1/2}^n\vee k, u_{j+1/2}^n \vee k, u_{j+3/2}^n \vee k \bigr) 
   	 \\& \quad  - 
	  \mathcal{G}_{j+1/2}^n \bigl( 
   u_{j-1/2}^n \wedge k, u_{j+1/2}^n \wedge k, u_{j+3/2}^n \wedge k \bigr). 
     \end{align*}
Thus, defining 
\begin{align*} 
 \tilde{G}_j^n (r,s,k) & := |r-k|   \Phi_+' \bigl( \alpha_{j-1/2}^n \bigr)  + 
 |s-k |  \Phi_-' \bigl( \beta_{j+1/2}^n  \bigr)  \\
 & \quad -  \frac{1}{\Delta x } \bigl(|A(s)-A(k)|-|A(r)-A(k)|\bigr), 
 \end{align*}  
 we can  write 
 \begin{align} \label{scheme102} 
 \begin{split} 
& \bigl| \bar{u}_{j+1/2}^{n+1} -k + 
  \lambda k  \bigl[ \Delta_+  \Phi_+' \bigl( \alpha_{j-1/2}^n \bigr) 
     +\Delta_+   \Phi_-' \bigl( \beta_{j+1/2}^n \bigr) \bigr] \bigr| \\
 &     \leq \bigl|  u_{j+1/2}^n -k  \bigr|  - \lambda 
      \Delta_+ \tilde{G}_j^n \bigl(  u_{j-1/2}^n, u_{j+1/2}^n,k \bigr).
      \end{split} 
 \end{align}
 On the other hand, 
 \begin{align} \label{scheme103} 
 \begin{split} 
& \bigl| \bar{u}_{j+1/2}^{n+1} -k + 
  \lambda k  \bigl[ \Delta_+  \Phi_+' \bigl( \alpha_{j-1/2}^n \bigr) 
     +\Delta_+   \Phi_-' \bigl( \beta_{j+1/2}^n \bigr) \bigr] \bigr| \\
     & \geq \bigl| u_{j+1/2}^{n+1} -k \bigr| + \sgn \bigl( u_{j+1/2}^{n+1} -k\bigr) 
     \lambda k  \bigl[ \Delta_+  \Phi_+' \bigl( \alpha_{j-1/2}^n \bigr) 
     +\Delta_+   \Phi_-' \bigl( \beta_{j+1/2}^n \bigr) \bigr].
     \end{split} 
     \end{align} 
Combining \eqref{scheme102} and \eqref{scheme103}, we arrive at the 
 ``cell entropy inequality''     
\begin{align} \label{scheme104} 
\begin{split} 
& \bigl| u_{j+1/2}^{n+1}-k \bigr| 
 -\bigl| u_{j+1/2}^{n}-k \bigr|  + \lambda \Delta_+ 
     \tilde{G}_j^n \bigl(  u_{j-1/2}^n, u_{j+1/2}^n,k \bigr) \\ &  
+ \sgn \bigl( u_{j+1/2}^{n+1} -k\bigr) 
     \lambda k  \bigl[ \Delta_+  \Phi_+' \bigl( \alpha_{j-1/2}^n \bigr) 
     +\Delta_+   \Phi_-' \bigl( \beta_{j+1/2}^n \bigr) \bigr]  \leq 0. 
\end{split} 
\end{align}  
We now basically establish convergence to a solution that 
 satisfies \eqref{entrineq} by a Lax-Wendroff-type argument. 
 Now, multplying the $j$-th inequality in \eqref{scheme104} by 
  $\smash{ \int_{I_j} \varphi(x,t_n) \, \mathrm{d} x}$, where  
$I_j := [x_j, x_{j+1}]$ and $\varphi$ is a suitable smooth, non-negative 
test function, and summing the results over $j \in \mathbb{Z}$, we obtain an 
 inequality of the type $E_1+ E_2 + E_3 \leq 0$, where we define 
 \begin{align*} 
 E_1 & := \sum_{n=0}^{N-1} 
  \sum_{j \in \mathbb{Z}} 
   \bigl(  \bigl| u_{j+1/2}^{n+1}-k \bigr| 
 -\bigl| u_{j+1/2}^{n}-k \bigr| \bigr) \int_{I_j} 
 \varphi(x,t_n) \, \mathrm{d} x, \\
 E_2 & := \lambda k \sum_{n=0}^{N-1} 
  \sum_{j \in \mathbb{Z}} \sgn \bigl( u_{j+1/2}^{n+1} -k\bigr) 
      \Delta_+ \bigl(  \Phi_+' \bigl( \alpha_{j-1/2}^n \bigr) 
     +  \Phi_-' \bigl( \beta_{j+1/2}^n \bigr) \bigr) \int_{I_j} 
 \varphi(x,t_n) \, \mathrm{d} x, \\
 E_3 & := \lambda \sum_{n=0}^{N-1} 
  \sum_{j \in \mathbb{Z}}   \Delta_+ 
     \tilde{G}_j^n \bigl(  u_{j-1/2}^n, u_{j+1/2}^n,k \bigr)\int_{I_j} 
 \varphi(x,t_n) \, \mathrm{d} x.  
 \end{align*} 
 By a standard summation by parts and using that $\varphi$ has compact support, 
 we get 
  \begin{align*} 
 E_1 & = -\Delta t  \sum_{n=0}^{N-1} 
  \sum_{j \in \mathbb{Z}} 
     \bigl| u_{j+1/2}^{n+1}-k \bigr| 
  \int_{I_j} 
 \fracd{\varphi(x,t_{n+1})-\varphi(x,t_n)}{\Delta t} \, \mathrm{d} x 
 \end{align*} 
 and $\smash{E_3= E_3^1 + E_3^2}$, where 
 \begin{align*} 
 E_3^1
 &  :=   - \Delta t  \sum_{n=0}^{N-1} 
  \sum_{j \in \mathbb{Z}}  \Bigl( \bigl| u_{j-1/2}^n - k \bigr| \Phi_+' \bigl( 
   \alpha_{j-1/2}^n \bigr) + \bigl| u_{j+1/2}^n - k \bigr| \Phi_-' \bigl( 
   \alpha_{\beta+1/2}^n \bigr) \Bigr) \times \\
   & \qquad \qquad \times \int_{I_j} 
   \fracd{\varphi(x+\Delta x,t_n)-\varphi(x,t_n)}{\Delta x} \, \mathrm{d} x, \\
   E_3^2 & :=   -  \Delta t  \sum_{n=0}^{N-1} \sum_{j \in \mathbb{Z}} 
   \bigl| A\bigl( u_{j+1/2}^n\bigr) - A(k) \bigr| \times \\
   & \qquad \qquad \times  \int_{I_j} 
   \fracd{\varphi(x+\Delta x,t_n)-2\varphi(x,t_n)+\varphi(x-\Delta x, t_n,x)}{\Delta x^2} \, \mathrm{d} x.  
 \end{align*} 
Clearly, we have that 
\begin{align*} 
E_2= \Delta t \sum_{n=0}^{N-1}  E_{2,n}^1 +
 E_2^2,
\end{align*}  
  where 
\begin{align*} 
E_{2,n}^1 & :=   \frac{k}{\Delta x} 
  \sum_{j \in \mathbb{Z}} \sgn \bigl( u_{j+1/2}^n -k\bigr) 
       \bigl[ \Delta_+  \Phi_+' \bigl( \alpha_{j-1/2}^n \bigr) 
     +\Delta_+   \Phi_-' \bigl( \beta_{j+1/2}^n \bigr) \bigr] \int_{I_j} 
 \varphi(x,t_n) \, \mathrm{d} x, \\
E_2^2  & :=   \lambda \sum_{n=0}^{N-1} 
  \sum_{j \in \mathbb{Z}}  \bigl[  \sgn 
\bigl( u_{j+1/2}^{n+1} -k\bigr) -  \sgn 
\bigl( u_{j+1/2}^{n} -k\bigr)  \bigr]   \times \\& \qquad \times 
     k  \bigl[ \Delta_+  \Phi_+' \bigl( \alpha_{j-1/2}^n \bigr) 
     +\Delta_+   \Phi_-' \bigl( \beta_{j+1/2}^n \bigr) \bigr] \int_{I_j} 
 \varphi(x,t_n) \, \mathrm{d} x. 
\end{align*} 

In the remainder of the proof, we  will  appeal to 
  \eqref{scheme105} and \eqref{scheme106}, and assume that both 
   $\Phi_-$ and $\Phi_+$ are  smooth away from $v^*$.  Moreover, we know that 
    for each~$n$, the data $\smash{\{v_{j}^n\}_{j \in \mathbb{Z}} }$ 
     are monotone.
     Therefore we will utilize once again that for each fixed  
      $n \in \{ 0, \dots, N-1 \}$, there exists $k$ such that 
      $\smash{v^{n}_{k}<v^*\leq v^{n}_{k+1}}$ and using Lemma~\ref{position_of_v^*} 
      we know that $\smash{v^{n+1}_{\bar{k}}<v^*\leq v^{n+1}_{\bar{k}+1}}$ with 
      $\smash{\bar{k}\in\left\{k-1,k,k+1\right\}}$. 
       Thus, if $\mathcal{A}^n$, $\mathcal{B}^n$ and $\mathcal{C}^n$ are 
        the  sets defined in \eqref{usefulsets}, 
      we may rewrite $E^1_{2,n}$ as 
$$E^1_{2,n}=E^1_{2,\mathcal{A}^n}+E^1_{2,\mathcal{B}^n}+E^1_{2,\mathcal{C}^n},$$
where the subindex denotes the summation over~$j$ from  
 the sets $\mathcal{A}^n$, $\mathcal{B}^n$ and $\mathcal{C}^n$, respectively. 
         For $\smash{j \in \mathcal{A}^n}$ we note that
         \begin{align*} 
          \Delta_+ \Phi_+' \bigl( \alpha_{j-1/2}^n\bigr) +\Delta_+\Phi_{-}'
           \bigl(\beta^{n}_{j+1/2} \bigr)
          & =   \Delta_+ \Phi' \bigl(\alpha_{j-1/2}^n \bigr);   
           \end{align*}  
           using a Taylor expansion about $\smash{v^n_j}$ we can write 
         \begin{align*} 
          \Delta_+ \Phi_+' \bigl(\alpha_{j-1/2}^n \bigr) 
              = \frac{\Delta x}{2} \bigl( u_{j+1/2}^n 
            +  u_{j-1/2}^n \bigr)   \Phi'' \bigl( v_j^n 
            \bigr)  + \mathcal{O} (\Delta x^2).     
            	\end{align*}      
             Thus,
            we obtain 
            \begin{align*} 
             	E^1_{2,\mathcal{A}^n} = \sum_{j \in \mathcal{A}^n} 
	           \fracd{k}{2} \sgn  \bigl( u_{j+1/2}^n-k \bigr)  \bigl( u_{j+1/2}^n 
            +  u_{j-1/2}^n \bigr)   \Phi'' \bigl( v_j^n 
            \bigr)  \int_{I_j} 
 \varphi(x,t_n) \, \mathrm{d} x  + \mathcal{O} (\Delta x). 
	 \end{align*} 
	 Since $\TV (U^n)$ is uniformly bounded, this implies 
	 that 
	 \begin{align} \label{eq200}  
             	E^1_{2,\mathcal{A}^n} = k\sum_{j \in \mathcal{A}^n} 
	           \sgn  \bigl( u_{j+1/2}^n-k \bigr)   u_{j+1/2}^n 
               \Phi'' \bigl( v_j^n 
            \bigr)  \int_{I_j} 
 \varphi(x,t_n) \, \mathrm{d} x  + \mathcal{O} (\Delta x). 
	 \end{align} 
	Analogously, we obtain  
	 \begin{align} \label{eq201}  
             	E^1_{2,\mathcal{C}^n} = k\sum_{j \in \mathcal{C}^n} 
	           \sgn  \bigl( u_{j+1/2}^n-k\bigr)     u_{j+1/2}^n 
               \Phi'' \bigl( v_j^n 
            \bigr)  \int_{I_j} 
 \varphi(x,t_n) \, \mathrm{d} x  + \mathcal{O} (\Delta x). 
	 \end{align}
 Moreover, since $\Phi$ is smooth and  we are  inspecting the situation 
 near the extremum (i.e. $\Phi'(v^*)=0$), we can conclude using  Taylor expansions that
 \begin{align*}
 \Delta_+  \Phi_+' \bigl( \alpha_{j-1/2}^n \bigr)  
     +\Delta_+   \Phi_-' \bigl( \beta_{j+1/2}^n \bigr)     = \mathcal{O} 
     (\Delta x)\quad \text{for  $j\in \mathcal{B}^n$}.
\end{align*}
Since $\mathcal{B}^n$ is finite,   
 $\smash{E^1_{2, \mathcal{C}^n} = \mathcal{O} (\Delta x)}$. 
        Combining this with \eqref{eq200} and \eqref{eq201} we obtain  
       \begin{align*} 
             	E^1_{2,n} = k \sum_{j \in \mathbb{Z}} 
	           \sgn  \bigl( u_{j+1/2}^n-k \bigr)   u_{j+1/2}^n 
               \Phi'' \bigl( v_j^n 
            \bigr)  \int_{I_j} 
 \varphi(x,t_n) \, \mathrm{d} x  + \mathcal{O} (\Delta x). 
	 \end{align*} 
       On the other hand, a summation by parts yields that 
       \begin{align*} 
        E_2^2 & = - \lambda k 
        \sum_{n=0}^{N-1} \sum_{j \in \mathbb{Z}} 
          \sgn \bigl(u_{j+1/2}^{n+1}-k \bigr) 
          \biggl[  \Bigr\{ \Delta_+ \Phi_+' \bigl( \alpha_{j-1/2}^{n+1} \bigr) 
            + \Delta_- \Phi_+' \bigl( \beta_{j+1/2}^{n+1} \bigr) \\
            & \quad \qquad 
             - \Delta_+ \Phi_+' \bigl( \alpha_{j-1/2}^{n} \bigr) 
            - \Delta_- \Phi_+' \bigl( \beta_{j+1/2}^{n} \bigr) \Bigr\}  \int_{I_j} 
 \varphi(x,t_n) \, \mathrm{d} x \\
 & \quad \qquad 
 +            \bigl[  \Delta_+ \Phi_+' \bigl( \alpha_{j-1/2}^{n} \bigr) 
             +\Delta_+ \Phi_+' \bigl( \beta_{j+1/2}^{n} \bigr) \bigr] 
              \int_{I_j} \bigl( \varphi(x, t_{n+1})- \varphi(x, t_n) \bigr) \, 
               \mathrm{d}x\biggr]. 
  \end{align*} 
Using arguments similar to those of the discussion of $E_2^1$ and  Lemma~\ref{position_of_v^*},  we see that the 
 expression in curled brackets is $\mathcal{O} ( \Delta x^2)$, 
  and finally $\smash{   E_2^2 = \mathcal{O} (\Delta x)}$, so that 
  \begin{align*} 
             	E_2  = k \Delta t \sum_{n=0}^{N-1} \sum_{j \in \mathbb{Z}} 
	           \sgn  \bigl( u_{j+1/2}^n-k)   u_{j+1/2}^n 
               \Phi'' \bigl( v_j^n 
            \bigr)  \int_{I_j} 
 \varphi(x,t_n) \, \mathrm{d} x  + \mathcal{O} (\Delta x). 
	 \end{align*}
	 A treatment similar to that of $E_2^1$ yields 
	 \begin{align*} 
	   E_3^1
    =   - \Delta t  \sum_{n=0}^{N-1} 
  \sum_{j \in \mathbb{Z}}   \bigl| u_{j+1/2}^n - k \bigr| \Phi' \bigl( 
   v_{j}^n \bigr) 
     \int_{I_j} 
   \fracd{\varphi(x+\Delta x,t_n)-\varphi(x,t_n)}{\Delta x} \, \mathrm{d} x\,+\mathcal{O}(\Delta x).
   \end{align*} 
   Since $\varphi$ is smooth, we may state the  inequality 
   $\smash{-E_1  -E_2 -E_3^1 - E_3^2 \geq 0}$ as 
   \begin{align*} 
    \iint_{\Pi_T} & 
     \Bigl\{ |u^{\Delta}-k | \bigl( \varphi_t 
      + \Phi' (v^{\Delta}) \varphi_x \bigr) 
       - \sgn (u^{\Delta}-k) u^{\Delta} k \Phi'' (v^{\Delta}) \\ & + 
        \bigl| A(u^{\Delta})-A(k) \bigr| \varphi_{xx} \Bigr\} 
         \, \mathrm{d}x \,\mathrm{d} t \geq - C_{11}   \Delta x 
         \end{align*} 
%
with a constant $C_{11}$ that is independent of $\Delta$. Taking $\Delta \to 0$ we obtain that the 
   limit function~$u$ satisfies the entropy inequality~\eqref{entrineq} for almost all $k\in\mathbb{R}$. 
    To prove that   \eqref{entrineq} 
     is valid for  all $k\in\mathbb{R}$ we may proceed according to 
    Lemmas~4.3 and~4.4 of \cite{krtl1}.  
%
\end{proof}

\begin{remark} Theorem~\ref{th4.1} implies, of course, that an entropy solution exists.  
  An inspection of the proofs of Lemmas~\ref{LC_respect_t} and~\ref{L_infinity_u} 
   reveals that the $L^{\infty}$ bound for $u$ is actually independent of~$T$. 
    Thus, even though our analysis is limited to domains~$\Pi_T$ with a finite final time~$T$, 
     we can say that entropy solutions of \eqref{goveq}, \eqref{initcond} do not blow up 
     in any finite time.  
\end{remark}

\section{Numerical examples} \label{sec:num} The examples presented here illustrate 
 the qualitative behavior of entropy solutions of the 
 initial value problem  \eqref{goveq}, \eqref{initcond}  and  
   the convergence properties of the numerical scheme. For the first purpose, 
    we select a relatively fine discretization and present the  corresponding 
    numerical solutions 
     as three-dimensional succesions of profiles at selected times or contour plots that  should almost 
     be free of numerical artefacts, while the convergence properties of the 
     scheme are demonstrated by error histories in some examples.  For all  numerical examples we 
      specify $\Delta x$ and use $\mu = \Delta t/\Delta x^2=0.1$, i.e., 
       $\Delta t = 0.1 \Delta x^2$. 
       
\begin{figure}[t]
\begin{center} 
\includegraphics[width=0.9\textwidth]{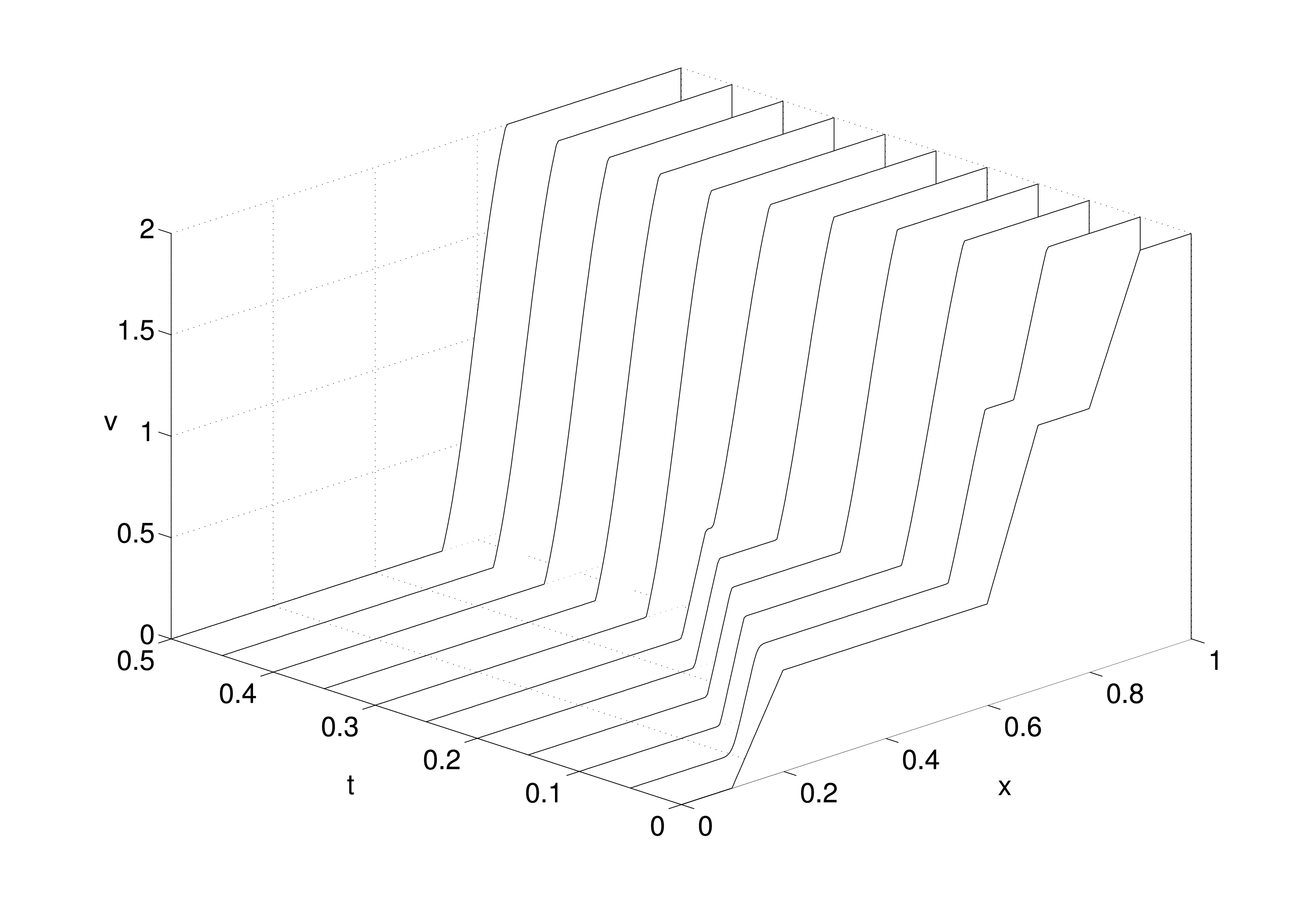} \\
\includegraphics[width=0.9\textwidth]{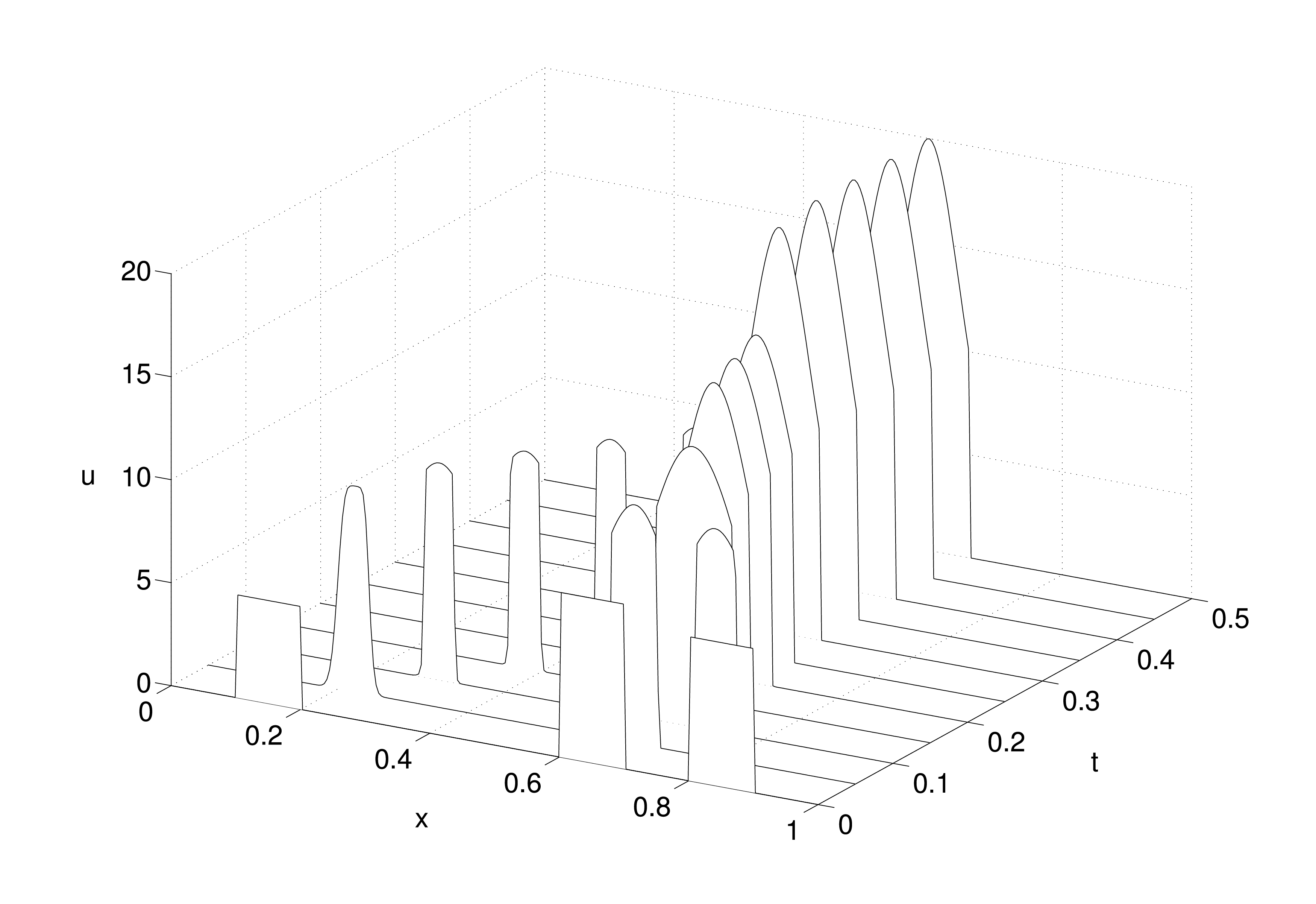}
\end{center} 
\caption{Example~1: Numerical approximation of~$v$ (top) and corresponding 
 approximation of~$u$ (bottom), obtained via \eqref{comb_scheme} with $\Delta x=0.001$.\label{Fig.1}}
\end{figure}

\begin{table}[t]
\begin{center}
\begin{tabular}{cccccccccc}
\hline
 $\Delta x$  & $e^{t_{1}}_v$& ${}^{\mathrm{conv.}}_{\mathrm{rate}}$&  $e^{t_{2}}_v$& ${}^{\mathrm{conv.}}_{\mathrm{rate}}$ & $e^{t_{1}}_u$ & ${}^{\mathrm{conv.}}_{\mathrm{rate}}$ & $e^{t_{2}}_u$ & 
 ${}^{\mathrm{conv.}}_{\mathrm{rate}}$ $\vphantom{\int_{X_x}^X}$ \\
\hline
0.020	 & 0.239 &   -   & 0.317 &   -   & 0.915 &   -   & 0.695 &   -\\
0.010    & 0.133 & 0.845 & 0.146 & 1.122 & 0.513 & 0.834 & 0.442 & 0.655\\
0.005    & 0.061 & 1.135 & 0.069 & 1.070 & 0.246 & 1.062 & 0.200 & 1.144\\
0.004    & 0.048 & 1.018 & 0.054 & 1.090 & 0.181 & 1.369 & 0.164 & 0.891\\
0.002    & 0.021 & 1.168 & 0.024 & 1.161 & 0.082 & 1.150 & 0.073 & 1.163\\
0.001	 & 0.008 & 1.360 & 0.009 & 1.399 & 0.036 & 1.167 & 0.032 & 1.200$\vphantom{\int_X}$\\
\hline
\end{tabular}

\vspace*{2mm} 

\end{center}
\caption{\label{table1}Example 1: Numerical error at $t_1=0.1$ and $t_2=0.25$.}
\end{table}

\begin{figure}[t] 
\begin{center} 
\includegraphics[width=0.95\textwidth]{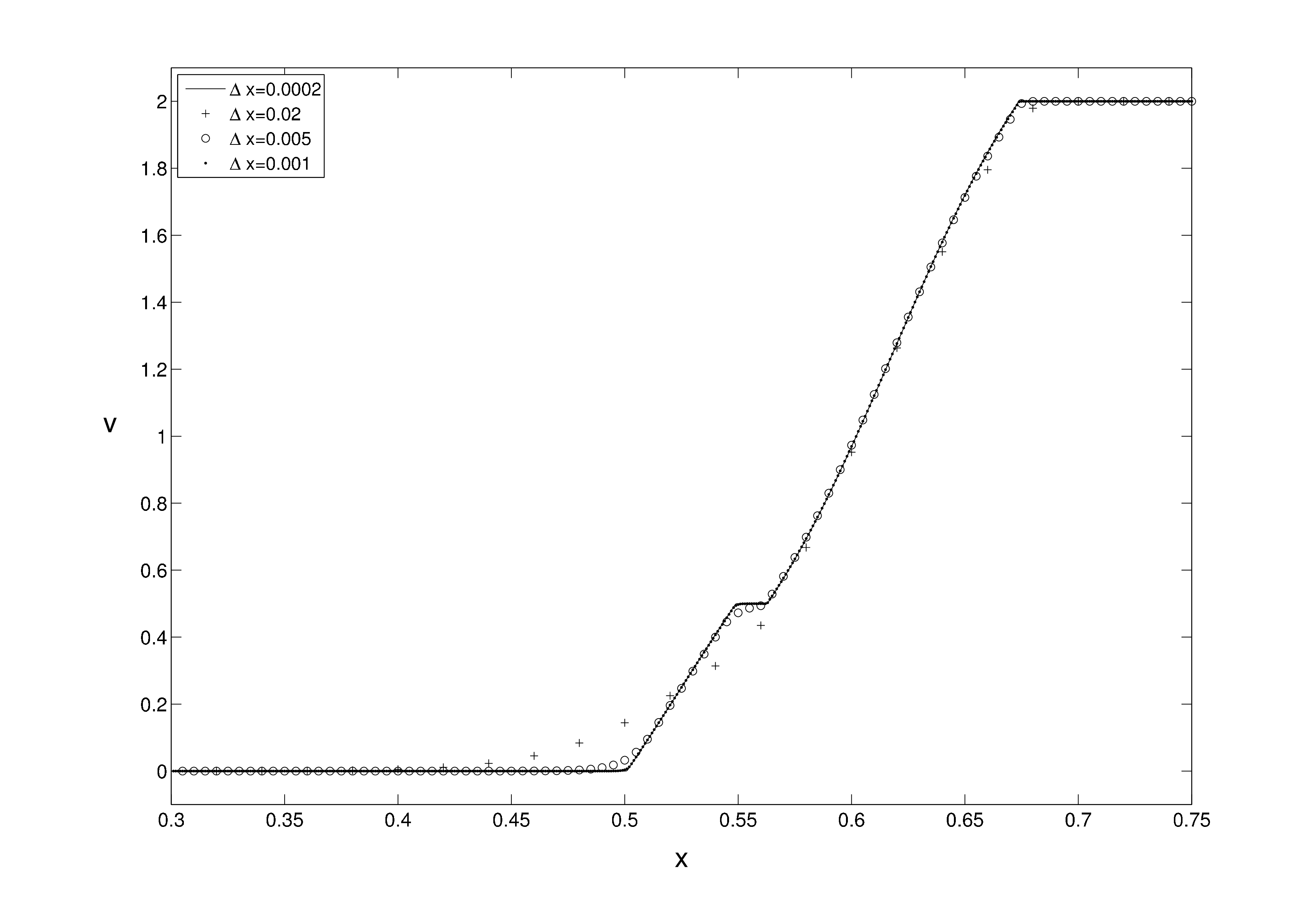} \\
\includegraphics[width=0.95\textwidth]{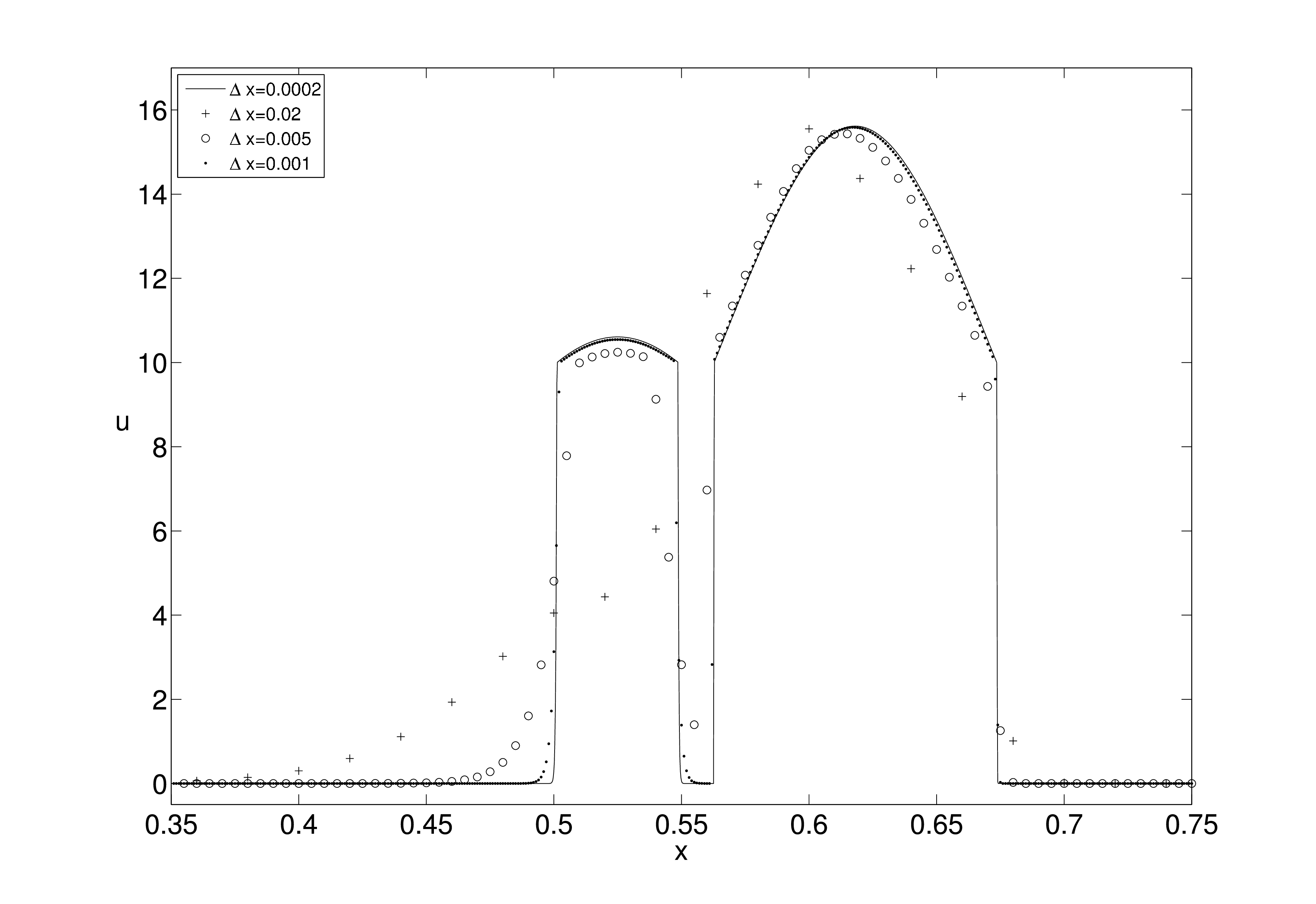}
\end{center} 
\caption{Example~1: Numerical  approximation of~$v$ (top) and~$u$ (bottom) for several mesh sizes at $t=0.25$.\label{comparison_u_05_ex1a}}
\end{figure}

\subsection{Example 1}
In Example~1 we calculate the numerical solution of 
 \eqref{goveq}, \eqref{initcond}  for 
   $ \Phi(v)=-(1-v)^2$ and the degenerating 
  diffusion coefficient 
  \begin{align*}
 A(u)=\begin{cases}
       0 & \text{for $0\leq u\leq 10$,} \\
0.1(u-10) & \text{for $u> 10$.} 
      \end{cases}
\end{align*}
The initial datum is given by
\begin{align*}
 u_0(x)=\begin{cases}
         5 \quad\text{for $0.1\leq x\leq 0.2$,}\quad  & 7\quad\text{for $0.8\leq x\leq 0.9$,} \\
         8 \quad\text{for $0.6\leq x\leq 0.7$,} & 0\quad\text{otherwise.}
        \end{cases}
\end{align*}

Note that $C_0=2$ in Example~1, where $C_0$ is defined in \eqref{Cdef}, 
and $v^*=C_0/2=1$, so that the function~$\Phi$ corresponds to \eqref{phiez}, 
where the constant of integration is $-1$, and that $u_0$~is chosen such that 
 \eqref{initial_data_stability} is satisfied. 
Moreover, in our case $\Phi''(v) = -2 <0$, and $\Phi(0) = \Phi(C_0) = -1$. 
 Nagai and Mimura \cite{nagai3} show that under these conditions, and 
 for the integrated diffusion coefficient given by \eqref{aum}, the 
 solution  converges in time  to a compactly supported, stationary travelling-wave 
 solution,  which represents the aggregated group of individuals and is 
  defined by the time-independent version  of~\eqref{goveq}.  
    
In Figure~\ref{Fig.1} we show the numerical approximations for~$v$ and~$u$ 
    for  $0 \leq t \leq 0.5$ and for $\Delta x=0.001$. As predicted, for each fixed  time 
 the  data $\smash{\{v_j^n\}}$ are monotonically increasing, and   
   the  numerical solution for~$u$ indeed displays the aggregation phenomenon, and  terminates 
    in a stationary profile, even though the assumptions on $A(\cdot)$ 
     stated in \cite{nagai3} are not satisfied here. This supports the  conjecture that 
     a similar travelling wave analysis can also be conducted in the 
     present strongly degenerate case, to which we will come back in a separate paper. 

In Table~\ref{table1} we show the error at $t_1=0.1$ and $t_2=0.25$ in the $L^1$ norm for $u$ (denoted as $\smash{e^{t_{i}}_u}$, $i=1,2$) and in the $L^\infty$ norm for $v$ (denoted as 
 $\smash{e^{t_{i}}_v}$, $i=1,2$), where we take as a reference  the solution calculated with $\Delta x=0.0002$. 
 We find an experimental rate of convergence in both cases greater than one. 
  For small~$\Delta x$  this  behavior is possibly related to  the proximity of the reference solution. 
   One should expect a real order of convergence 
  at most one  since the $v$-scheme is monotone.

In Figure~\ref{comparison_u_05_ex1a}  we compare the 
 numerical approximations for~$v$ and~$a$  
   for different mesh sizes at the simulated time $t=0.25$.

\begin{figure}[t] 
\begin{center} 
\includegraphics[width=0.9\textwidth]{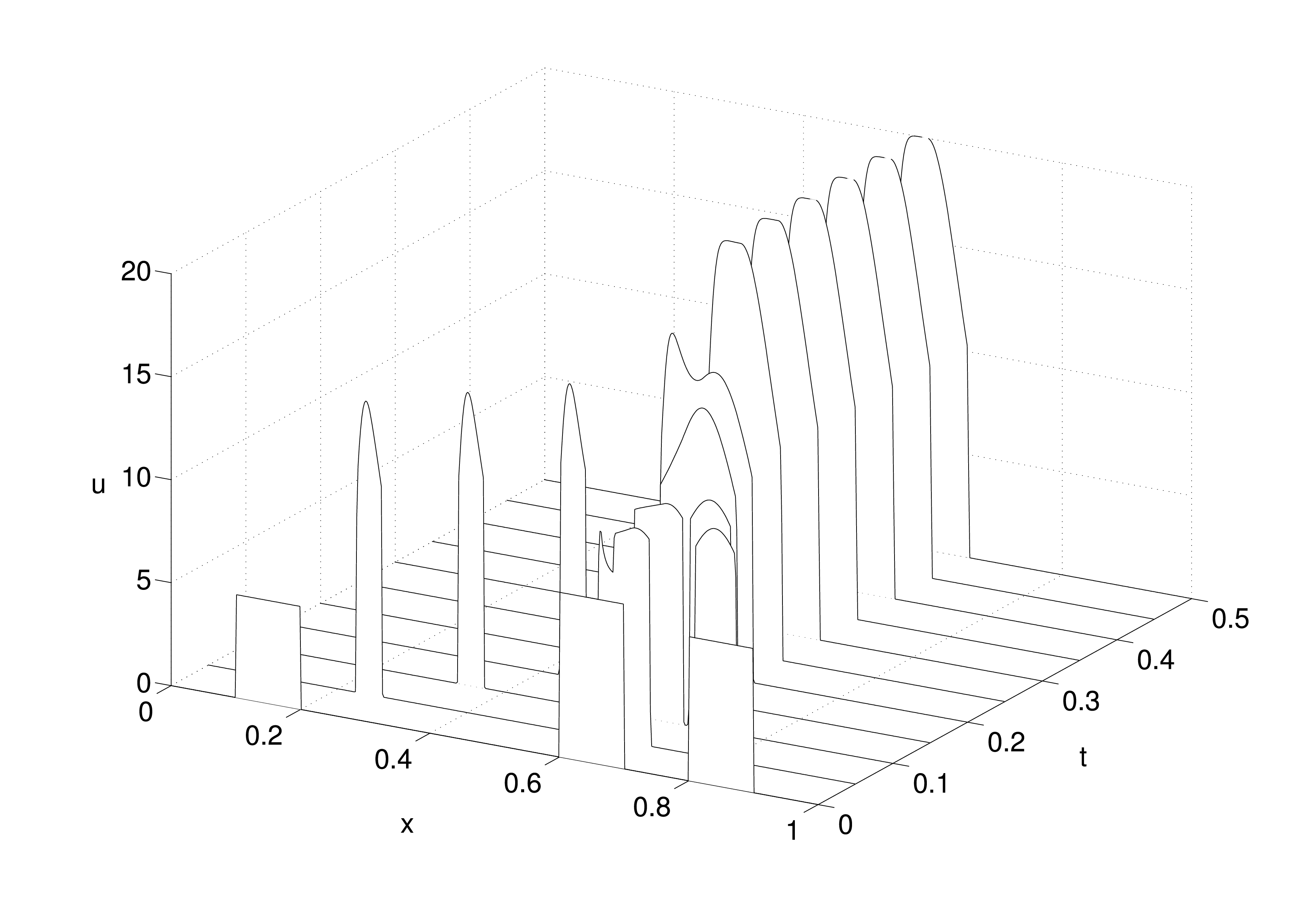}
\end{center} 
\caption{Example~2: Numerical approximation of~$u$, obtained via \eqref{comb_scheme} with $\Delta x=0.001$.\label{Fig.3}}
\end{figure} 

\subsection{Example 2} This example   represents a 
slight modification of Example~1, namely we choose $A(\cdot)$ and $u_0$ as in Example~1 but use 
\begin{align*}
 \Phi(v)=\begin{cases}
-(1-v)^8& \textrm{for $0\leq v\leq 1$,} \\
-(1-v)^2&\textrm{for $v>1$.} 
         \end{cases}
\end{align*}
The function~$\Phi$ 
  has its maximum in $v^*=1$  and satisfies 
   $\Phi(0)= \Phi(C_0)$, as in Example~1, so we expect to see 
    an aggregation phenomenon and the formation of a stationary travelling 
     wave, even though   here  $\Phi$ is not  symmetric with respect to $v^*$.  
   In Figure~\ref{Fig.3} we show the numerical approximation of~$u$  
  for  $0 \leq t \leq 0.5$ and  $\Delta x=0.001$. The solution behavior is 
  similar to that of~Example 1, but the shapes of the ``peaks'' 
   are slightly different, in particular the final ``peak'' is asymmetric.

\subsection{Example 3}
We now choose $\Phi$ and $u_0$ as in Example 1, 
 but define $A(\cdot)$ by  
\begin{align*}
 A(u)=\begin{cases}
       0.05u & \text{for $0\leq u\leq 5$,} \\
0.25 &\text{for $5<u\leq 10$,} \\
0.05u-0.25 &\text{for $u>10$.}  
      \end{cases}
\end{align*}
Figure~\ref{Fig.4} shows the results for $\Delta x=0.001$ and $t\in[0,0.5]$.
 Again, a stationary single-peak solution is forming, including  
   a jump between 
   $u=5$ and $u=10$, in agreement with the flatness of $A(u)$ 
    for $u \in [5,10]$. Table~\ref{table5}  shows  the error of the 
     approximations of~$v$ and~$u$ at $t_1=0.1$ and $t_2=0.25$. The reference  solution has been calculated with $\Delta x=0.0002$. The observed rate  of convergence is again one.

\begin{figure}[t] 
\begin{center} 
\includegraphics[width=0.9\textwidth]{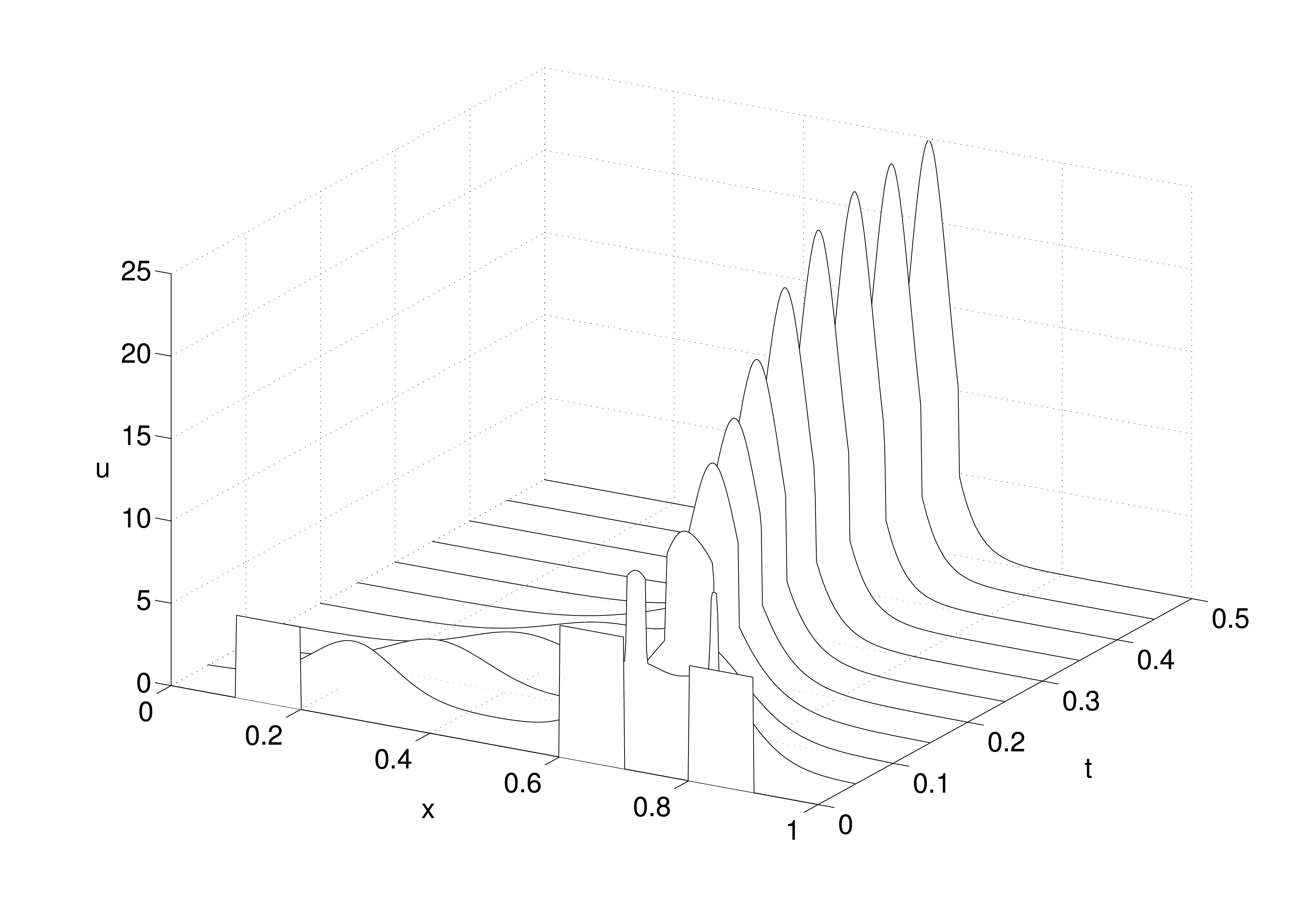}
\end{center} 
\caption{Example~3: Numerical approximation of~$u$, obtained via \eqref{comb_scheme} for $\Delta x=0.001$.\label{Fig.4}}
\end{figure}

\begin{table}[t]
\begin{center}
\begin{tabular}{cccccccccc}
\hline
 $\Delta x$  & $e^{t_{1}}_v$& ${}^{\mathrm{conv.}}_{\mathrm{rate}}$ &  $e^{t_{2}}_v$& ${}^{\mathrm{conv.}}_{\mathrm{rate}}$ & $e^{t_{1}}_u$ & ${}^{\mathrm{conv.}}_{\mathrm{rate}}$ & $e^{t_{2}}_u$ & ${}^{\mathrm{conv.}}_{\mathrm{rate}}$ $\vphantom{\int_{X_x}^X}$\\
\hline
0.020	 & 0.167 &   -   & 0.330 &   -   & 0.329 &   -   & 0.297 &   -\\
0.010    & 0.083 & 1.010 & 0.166 & 0.994 & 0.185 & 0.834 & 0.161 & 0.884\\
0.005    & 0.039 & 1.099 & 0.079 & 1.066 & 0.105 & 0.812 & 0.086 & 0.912\\
0.004    & 0.031 & 0.932 & 0.062 & 1.097 & 0.089 & 0.733 & 0.064 & 1.322\\
0.002    & 0.014 & 1.195 & 0.028 & 1.170 & 0.043 & 1.060 & 0.034 & 0.886\\
0.001	 & 0.006 & 1.165 & 0.010 & 1.414 & 0.021 & 1.056 & 0.016 & 1.115$\vphantom{\int_X}$\\
\hline
\end{tabular}

\vspace*{2mm} 

\end{center}
\caption{\label{table5}Example 3: Numerical error for~$u$ and~$v$ at $t_1=0.1$ and $t_2=0.25$.}
\end{table}

\subsection{Example 4}
In this example 
 we  utilize a flux function with several extrema given by  
 $\Phi(v)=-0.5(\cos(v\pi)+1)$ combined with  
 the integrated diffusion coefficient  
\begin{align*}
 A(u)=\begin{cases}
       0  &  \text{for $0\leq u\leq 10$,} \\
0.1(u-10) & \text{for $u> 10$} 
      \end{cases}
\end{align*}
and  the initial datum   
\begin{align*}
 u_0(x)=\begin{cases}
  10 \quad \text{for $x\in[0.05,0.15]$,}  & 9 \quad \text{for $x\in[0.6,0.7]$,} \\
  14 \quad \text{for  $x\in[0.3,0.5]$,} & 8 \quad \text{for $x\in[0.9,1]$,} \\
   & 0 \quad \text{otherwise.}
 \end{cases}
\end{align*}

 \begin{figure}[t] 
\begin{center} 
\includegraphics[width=0.9\textwidth]{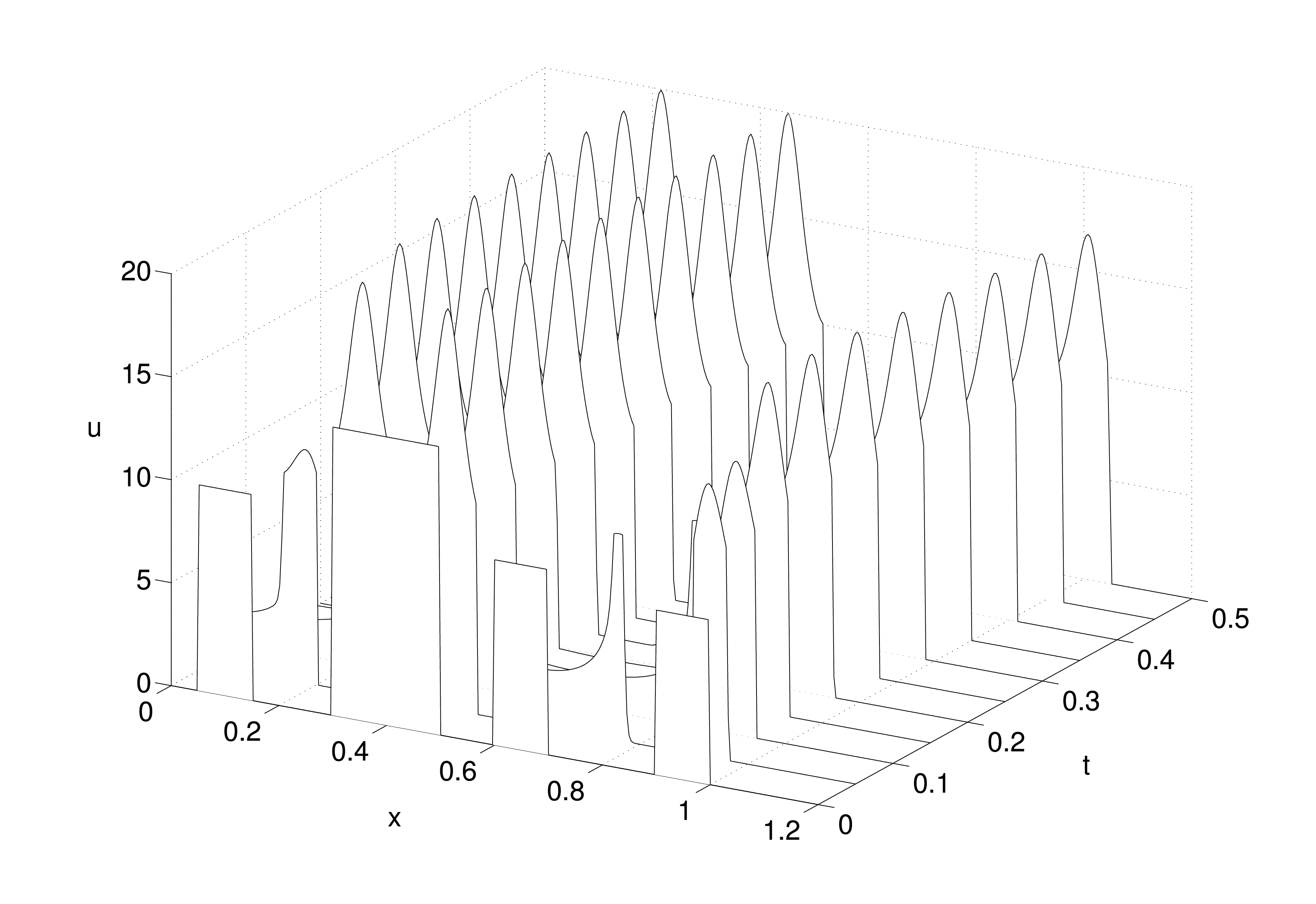}
\end{center} 
\caption{Example 4: Numerical approximation of~$u$, obtained via \eqref{comb_scheme} for $\Delta x=0.001$.\label{Ex3b}}
\end{figure} 

\begin{figure}[t] 
\begin{center} 
\includegraphics[width=0.9\textwidth]{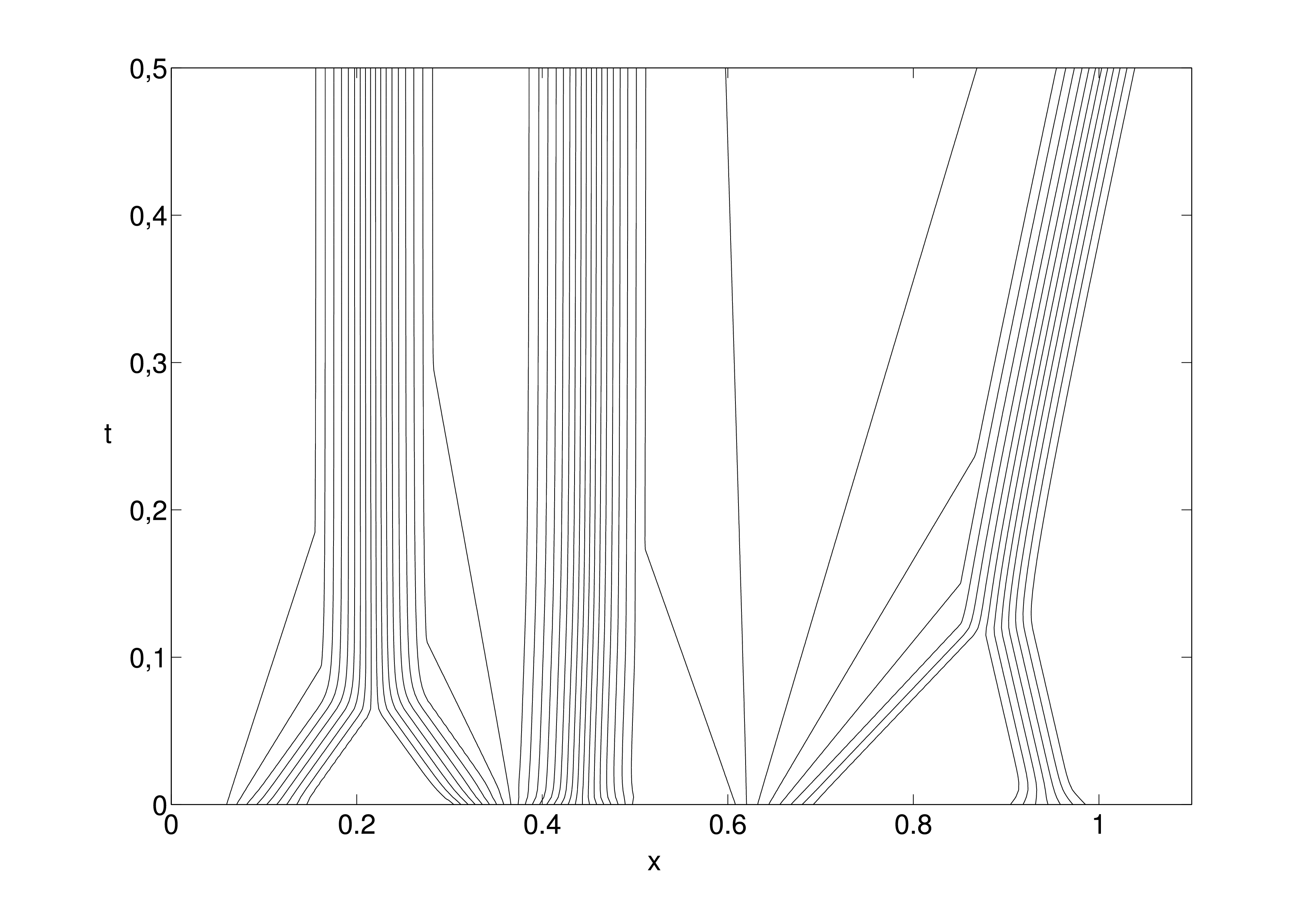}
\end{center} 
\caption{Example~4: Contour lines of the numerical approximation of~$v$ for $\Delta x=0.001$.\label{Ex3a_contour}}
\end{figure}

The result is shown in Figure~\ref{Ex3b} for $\Delta x=0.001$. 
We observe  the formation of three groups, but the  third  moves to the right ``looking for more'' mass since it  is not a full state, in the sense of the Nagai and Mimura \cite{nagai3} condition for the formation of stationary travelling waves. In addition to Figure~\ref{Ex3b} we show 
 in Figure~\ref{Ex3a_contour} a contour plot of the numerical approximation of~$v$ 
 for this example. The contour lines of~$v$ correspond to trajectories of ``individuals''. 
 Table~\ref{table3} shows the error for~$v$ and~$u$ taking as a reference the solution calculated with $\Delta x=0.0002$. We again find the order of convergence predicted for monotone schemes. 

\begin{table}[t]
\begin{center}
\begin{tabular}{cccccccccc}
\hline
 $\Delta x$  & $e^{t_{1}}_v$& ${}^{\mathrm{conv.}}_{\mathrm{rate}}$ &  $e^{t_{2}}_v$& ${}^{\mathrm{conv.}}_{\mathrm{rate}}$ & $e^{t_{1}}_u$ & ${}^{\mathrm{conv.}}_{\mathrm{rate}}$ & $e^{t_{2}}_u$ & ${}^{\mathrm{conv.}}_{\mathrm{rate}}$
$\vphantom{\int_{X_x}^X}$ \\ \hline
0.020	 & 0.361 &   -   & 0.381 &   -   & 1.420 &   -   & 1.244 &   -\\
0.010    & 0.189 & 0.933 & 0.201 & 0.923 & 0.892 & 0.671 & 0.709 & 0.811\\
0.005    & 0.095 & 0.992 & 0.101 & 0.994 & 0.509 & 0.809 & 0.356 & 0.993\\
0.004    & 0.080 & 0.771 & 0.080 & 1.048 & 0.398 & 1.101 & 0.262 & 1.374\\
0.002    & 0.042 & 0.939 & 0.040 & 0.981 & 0.216 & 0.883 & 0.145 & 0.857\\
0.001	 & 0.019 & 1.130 & 0.020 & 1.000 & 0.104 & 1.047 & 0.072 & 1.000 $\vphantom{\int_X}$\\
\hline
\end{tabular}

\vspace*{2mm}

\end{center}
\caption{\label{table3}Example 4: Numerical error for~$u$ and~$v$ at $t_1=0.1$ and $t_2=0.25$.}
\end{table}

\subsection{Example 5}
Here we calculate the numerical approximation of~$u$ for $A(\cdot)$ as in Example 4, but with~$\Phi$ 
 and~$u_0$ given by
 the respective equations 
 \begin{gather*}
 \Phi(v)=\begin{cases}
-0.5(\cos(v\pi)+1)&\text{for $0\leq v\leq 2$,}\\
(v-2)^2-1&\text{for $v> 2$,} 
         \end{cases} \\
 u_0(x)=\begin{cases}
  14 \quad \text{for $x\in[0.15,0.3]$,}  & 18 \quad \text{for $x\in[0.8,0.95]$,} \\
  17 \quad \text{for  $x\in[0.6,0.7]$,} & 
    0 \; \, \quad \text{otherwise.}
 \end{cases}
\end{gather*}

\begin{figure}[t] 
\begin{center} 
\includegraphics[width=0.9\textwidth]{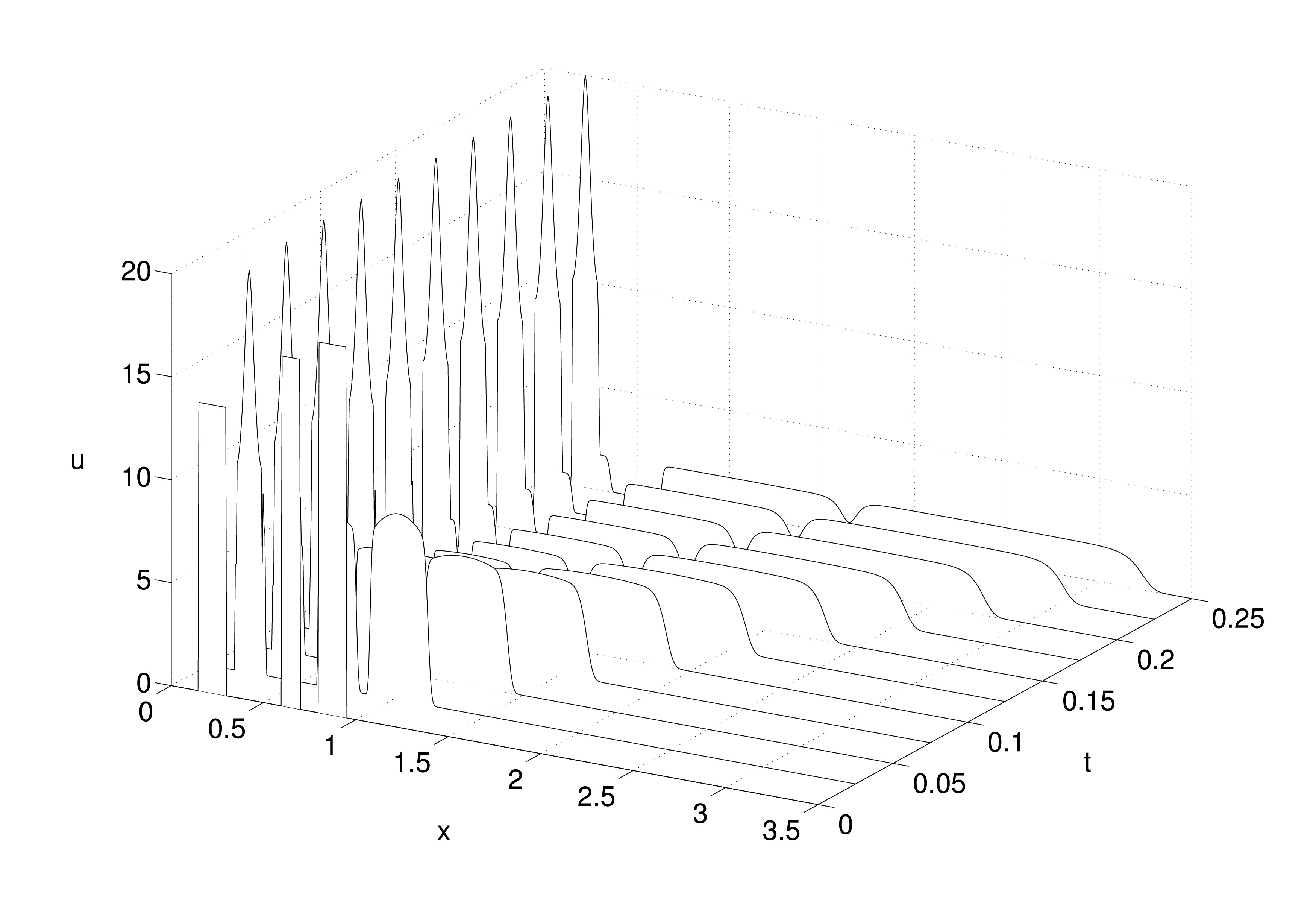}
\end{center} 
\caption{Example~5: Numerical approximation of~$u$, obtained via \eqref{comb_scheme} for $\Delta x=0.001$.\label{Ex4}}
\end{figure} 

In Figure \ref{Ex4} we show the result for $\Delta x=0.001$. We see that the spare mass (i.e. the mass that can not get in the first group) ``dilutes'' to the right. 

\subsection{Example 6}
We consider now the same initial data and parabolic term $A(\cdot)$ as in 
Example~1, but employ  a function~$\Phi$  with several extrema given by $\Phi(v)=-0.5\left(\cos(2\pi v)+1\right)$. Accordingly with the results of Example 1 we expect a steady state consisting of two traveling waves since $\Phi(0)=\Phi(C_0)$. Figure \ref{Ex6} shows the numerical result for~$u$ for $\Delta x=0.001$, which confirm our claim. 

\begin{figure}[t] 
\begin{center} 
\includegraphics[width=0.9\textwidth]{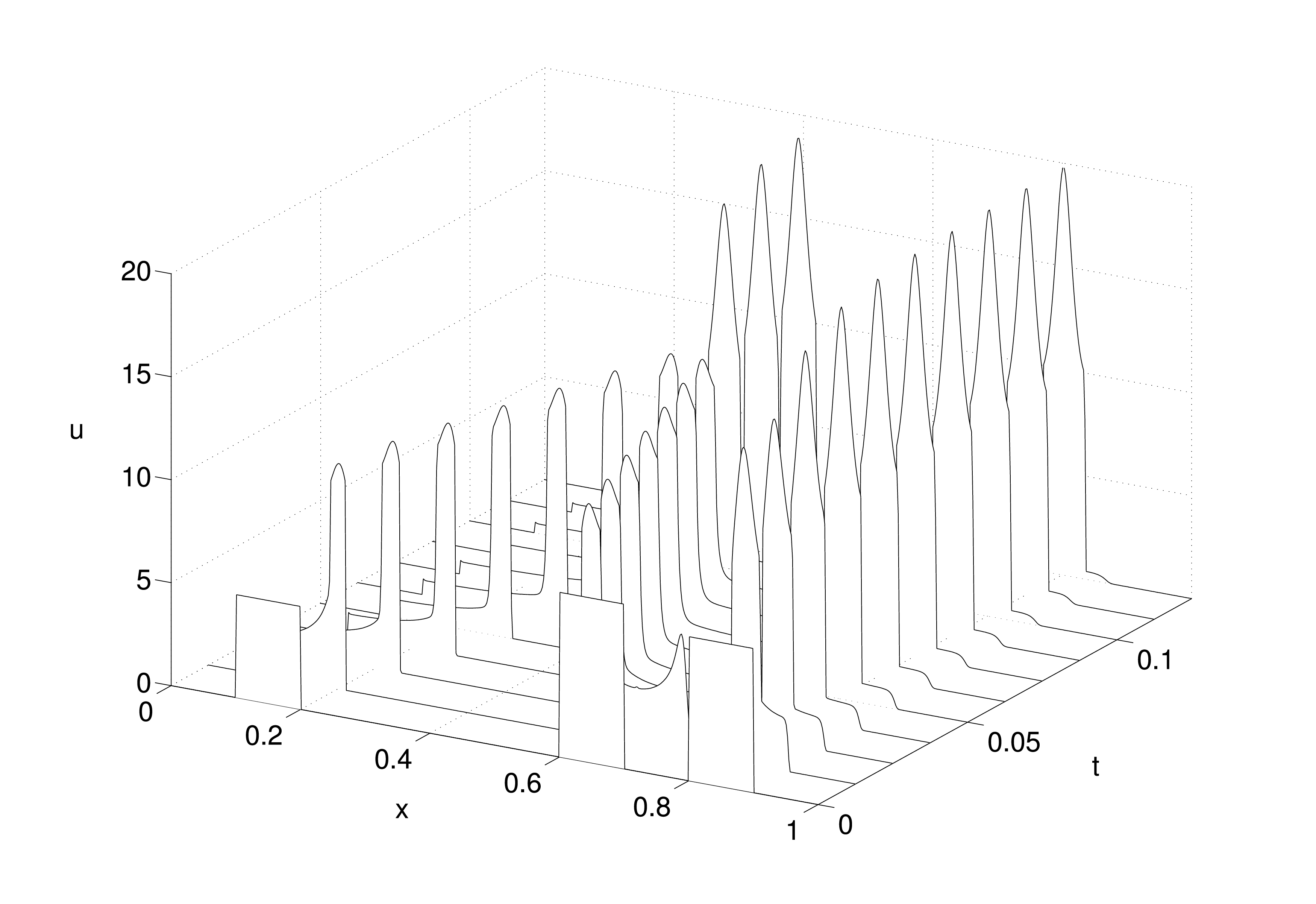}
\end{center} 
\caption{Example~6: Numerical approximation of~$u$, obtained via \eqref{comb_scheme} for $\Delta x=0.001$.\label{Ex6}}
\end{figure} 

\section*{Acknowledgements} FB acknowledges support by CONICYT fellowship.  
RB acknowledges  support by Fondecyt project 1090456, Fondap in
Applied Mathematics, project 15000001, and BASAL project CMM,
Universidad de Chile and Centro de Investigaci\'{o}n en
Ingenier\'{\i}a Matem\'{a}tica (CI$^{\mathrm{2}}$MA), Universidad de
Concepci\'{o}n.

\end{document}